\documentclass[11pt, a4paper]{amsart}
\usepackage{amsfonts,amssymb,amsmath, amsthm}
\usepackage{mathrsfs}
\usepackage{lmodern}
\usepackage{latexsym}
\usepackage{bm}
\usepackage{xcolor}

\providecommand{\norm}[1]{\left\lVert#1\right\rVert}
\newtheorem{thm}{Theorem}[section]
\newtheorem{lem}[thm]{Lemma}

\theoremstyle{definition}
\newtheorem{defn}[thm]{Definition}
\newtheorem{rem}[thm]{Remark}

\numberwithin{equation}{section}
\usepackage[plainpages=false,pdfpagelabels,backref=page,citecolor=red]{hyperref}
\usepackage{xcolor}
\hypersetup{
colorlinks,
linkcolor={cyan!90!black},
citecolor={magenta},
urlcolor={green!40!black}
} 
\newcommand{\R}{\mathbb{R}}
\newcommand{\IC}{\mathbb{C}}

\newcommand{\cH}{\mathcal{H}}

\newcommand{\cF}{\mathcal{F}} 

\providecommand{\abs}[1]{\left\lvert#1\right\rvert}

\newcommand{\loc}{\operatorname{loc}}
\renewcommand{\L}{\operatorname{L}} 
\renewcommand{\H}{\operatorname{H}}
\newcommand{\C}{\operatorname{C}} 
\renewcommand{\H}{\operatorname{H}} 
\newcommand{\W}{\operatorname{W}}
\newcommand{\Hdot}{\dot{\H}\protect{\vphantom{H}}} 
\newcommand{\E}{\mathsf{E}} 

\newcommand{\ree}{{\mathbb{R}^{n+1}}}

\newcommand{\gradx}{\nabla_x}

\renewcommand{\div}{\operatorname{div}}

\newcommand{\divx}{\div_x}
\newcommand{\dhalf}{D_t^{1/2}} 
\newcommand{\HT}{H_t} 

\renewcommand{\i}{\mathrm{i}} 
\renewcommand{\d}{\, \mathrm{d}} 
\renewcommand\Re{\operatorname{Re}}

\newcounter{teller}

\newcommand{\cl}[1]{\overline{#1}} 
\DeclareMathOperator{\dom}{\mathsf{D}} 

\newcommand{\pd}{\partial}

\newcommand{\m}{m}

\def\Xint#1{\mathchoice
{\XXint\displaystyle\textstyle{#1}}%
{\XXint\textstyle\scriptstyle{#1}}%
{\XXint\scriptstyle\scriptscriptstyle{#1}}%
{\XXint\scriptscriptstyle%
\scriptscriptstyle{#1}}%
\!\int}
\def\XXint#1#2#3{{\setbox0=\hbox{$#1{#2#3}{%
\int}$ }
\vcenter{\hbox{$#2#3$ }}\kern-.6\wd0}}
\def\barint{\,\Xint -} 
\def\bariint{\barint_{} \kern-.4em \barint}
\def\bariiint{\bariint_{} \kern-.4em \barint}
\renewcommand{\iint}{\int_{}\kern-.34em \int} 
\renewcommand{\iiint}{\iint_{}\kern-.34em \int} 

\providecommand{\norm}[1]{\lVert#1\rVert}

\newcommand\U{\mathcal{U}}

\newcommand{\ru}{\mathrm{u}}

\title[Fractional powers of  parabolic operators]{On local regularity estimates for fractional powers of parabolic operators with time-dependent\\ measurable coefficients}

\address{Malte Litsg{\aa}rd \\Department of Mathematics, Uppsala University\\
S-751 06 Uppsala, Sweden}
\email{malte.litsgard@math.uu.se}

\address{Kaj Nystr\"{o}m\\Department of Mathematics, Uppsala University\\
S-751 06 Uppsala, Sweden}
\email{kaj.nystrom@math.uu.se}

\thanks{K.N was partially supported by grant  2017-03805 from the Swedish research council (VR)}

\author{Malte Litsg{\aa}rd and Kaj Nystr{\"o}m}

\date{\today}

\begin{document}

\newpage
\begin{abstract} We consider fractional operators of the form
$$\cH^s=(\partial_t  -\div_{x} ( A(x,t)\nabla_{x}))^s,\ (x,t)\in\mathbb R^n\times\mathbb R,$$
where $s\in (0,1)$ and $A=A(x,t)=\{A_{i,j}(x,t)\}_{i,j=1}^{n}$  is an accretive, bounded, complex, measurable, $n\times n$-dimensional matrix valued function. We study the fractional operators ${\cH}^s$ and their relation to the initial value problem
\begin{equation*}
 \begin{split}
(\lambda^{1-2s}\ru')'(\lambda) &=\lambda^{1-2s}\cH \ru(\lambda), \quad \lambda\in (0, \infty),  \\
  \ru(0) & = u,
 \end{split}
\end{equation*}
in $\mathbb R_+\times \mathbb R^n\times\mathbb R$. Exploring the relation, and making  the additional assumption that $A=A(x,t)=\{A_{i,j}(x,t)\}_{i,j=1}^{n}$ is real, we derive some local properties of solutions to the non-local Dirichlet problem
\begin{align*}
\cH^su=(\partial_t  -\div_{x} ( A(x,t)\nabla_{x}))^su&=0\mbox{ for $(x,t)\in \Omega \times J$},\notag\\
u&=f\mbox{ for $(x,t)\in \mathbb R^{n+1}\setminus (\Omega \times J)$}.
\end{align*}
Our contribution is that we allow for non-symmetric and time-dependent coefficients.
\end{abstract}

\maketitle

\section{Introduction and background}

Fractional powers of closed linear operators
in Hilbert and Banach spaces is an important and classical topic in operator theory with fundamental contributions attached to
Bochner, Balakrishnan, Komatsu and many other prominent researchers, see  \cite{bochner1949,feller1952,balakrishnan1959, hille1948, phillips1952,komatsu1969}. The construction and application of
fractional powers of sectorial operators, i.e. linear operators having $\mathbb R_-$  contained in their resolvent set, and fulfilling an additional resolvent estimate, have attracted much attention resulting in a substantial literature on the topic, see \cite{balakrishnan1960,komatsu1969} and the extensive treatments in
\cite{Kato,MS,Haase}.

More recently,  Caffarelli and Silvestre \cite{caffarelli2007} induced new energy into the field by noting that if $s\in (0,1)$, and if $\mathcal{U}$ solves
\begin{equation}
\partial_\lambda(\lambda^{1-2s}\partial_\lambda\mathcal{U})(\lambda,x) =-\lambda^{1-2s} \Delta_x \mathcal{U}(\lambda,x),\ \mathcal{U}(0,x)  = u(x),\  (\lambda,x)\in \mathbb R_+\times\mathbb R^n,
\label{DN2}
\end{equation}
where $u\in \dom (( - \Delta_x )^{s})$, then
\begin{align}\label{DN1}
-\lim\limits_{\lambda \rightarrow 0^+} \lambda^{1-2s} \partial_\lambda \mathcal{U}(\lambda,x)  = c_s( - \Delta_x )^{s}u (x), \qquad x\in\R^n.
\end{align}
In particular, the non-local fractional Laplace operator can be realized as a Dirichlet to Neumann map using an extension problem for an associated (local) linear degenerate elliptic equation. As one consequence, local properties of solutions to $( - \Delta_x )^{s}u=0$ in a domain $\Omega\subset\mathbb R^n$ can be deduced using corresponding results for linear degenerate elliptic equations.

In \cite{NS} and \cite{Stinga-Torrea-SIAM}, independently, the parabolic analogue of  the result of Caffarelli and Silvestre \cite{caffarelli2007} was discovered and it is proved that if $s\in (0,1)$, and if $\mathcal{U}$ now solves
\begin{equation}
\partial_\lambda(\lambda^{1-2s}\partial_\lambda\mathcal{U})(\lambda,x,t) =\lambda^{1-2s} (\partial_t- \Delta_x )\mathcal{U}(\lambda,x,t),\ \mathcal{U}(0,x,t)  = u(x,t),\  (\lambda,x,t)\in \mathbb R_+\times\mathbb R^n\times \R,
\label{DN2+}
\end{equation}
where $u\in \dom (( \partial_t- \Delta_x )^{s})$, then
\begin{align}\label{DN1+}
 -\lim\limits_{\lambda \rightarrow 0^+} \lambda^{1-2s} \partial_\lambda \mathcal{U}(\lambda,x,t)=c_s(\partial_t- \Delta_x )^{s}u (x,t), \qquad (x,t)\in\R^n\times\R.
\end{align}
In particular, also the non-local fractional heat operator can be realized as a Dirichlet to Neumann map using an extension problem, now for an associated (local) linear degenerate parabolic equation, see also \cite{B-DLC-S,GFT}. As one consequence, local properties of solutions to $( \partial_t- \Delta_x )^{s}u=0$ in a domain $\Omega\times J\subset\mathbb R^n\times \mathbb R$ can be deduced by developing corresponding results for linear degenerate parabolic equations. Note that if $\mathcal{U}$ is independent of $t$, then formally the equation in \eqref{DN2+} coincides with the equation in \eqref{DN2}.

Given $u$, the construction in \cite{NS,Stinga-Torrea-SIAM} of the solution
$\mathcal{U}$ to \eqref{DN2+}, as well as  the corresponding construction in the case of more general operators $\partial_t  +\mathcal{L}_x$, $\mathcal{L}_x=-\div_{x} ( A(x)\nabla_{x})$, with $A$ real, bounded, uniformly elliptic and symmetric, can be stated
\begin{align}\label{51intro+}
\mathcal{U}(\lambda,x,t)&=\frac{1}{\Gamma(s)}\int_0^\infty h_s({\lambda^2}/{4r})e^{-r(\partial_t  +\mathcal{L}_x)}u(x,t) \, \frac{\d r}{r},
\end{align}
where $h_s(\tau):=\tau^{s}e^{-\tau}$. We refer to \cite{Stinga-Torrea-SIAM,B-DLC-S} for this formula and its details. Using the ellipticity of $A$,  the semigroup $e^{-r\mathcal{L}_x}$ is well understood and facilitates estimates. Furthermore, one can deduce that
\begin{align}\label{51intro++}
\mathcal{U}(\lambda,x,t)&= \iint_{\mathbb R^{n}}\Gamma_{s,\lambda}(x,t,y,s)u(y,s)\, \d y\d s,
\end{align}
for a non-negative kernel $\Gamma_{s,\lambda}$ which can be computed explicitly based on the fundamental solution for $\partial_t  +\mathcal{L}_x$.  This analysis relies heavily on the fact that $A$ is independent of $t$ and that $\partial_t$ and $\mathcal{L}_x$ commute.  We  refer to \cite{GT,GT1}, and \cite{GFT}, for interesting accounts of the research in this direction, also  covering certain classes of strongly degenerate parabolic operators of Kolmogorov type.

In this paper we are interested in generalizations of \eqref{DN2+} and \eqref{DN1+}, with $( \partial_t- \Delta_x )$ and $( \partial_t- \Delta_x )^{s}$ replaced by
$(\partial_t  -\div_{x} ( A(x,t)\nabla_{x}))$ and $(\partial_t  -\div_{x} ( A(x,t)\nabla_{x}))^s$, respectively. Concerning $A(x,t)=\{A_{i,j}(x,t)\}_{i,j=1}^{n}$ we  assume only that the matrix $A(x,t)$ is complex, measurable,  bounded and accretive. In particular, in the case of real coefficients we are concerned with fractional powers of second order parabolic operators  allowing for non-symmetric and time-dependent coefficients. In this generality, the very definition of the operator, and its fractional powers, is in itself an issue which requires concepts and notions from operator theory. In our case, the essence is that
    $(\partial_t  -\div_{x} ( A(x,t)\nabla_{x}))$ can be realized as a  {maximal accretive} operator $(\cH,\dom)$ on a certain energy space modelled on $\sqrt{\partial_t- \Delta_x}$. As a
    consequence, the fractional power of $\cH$, $\cH^s$ for $s\in (0,1)$, which formally coincides with $(\partial_t  -\div_{x} ( A(x,t)\nabla_{x}))^s$, can be defined
    using for example the Balakrishnan representation, see \eqref{ba}. These observations serves as the starting point for our analysis. We refer to Section \ref{prelim} for precise definitions.

Our work,  and the definition of $\cH^s=(\partial_t  -\div_{x} ( A(x,t)\nabla_{x}))^s$,  is rooted in recent work of the second author, together with P. Auscher and M. Egert,  concerning boundary value problems for second order parabolic equations (and systems) of the form
    \begin{eqnarray}\label{eq1l}
    \partial_tu-\div_{\lambda,x}A(x,t)\nabla_{\lambda,x}u=0,
    \end{eqnarray}
    in the  upper-half parabolic space $\mathbb R_+^{n+2}:=\{(\lambda,x,t)\in \mathbb R\times \mathbb R^n\times \mathbb R:\ \lambda>0\}$, $n\geq 1$, with boundary determined by $\lambda=0$, assuming only bounded, measurable, uniformly elliptic and complex coefficients.  In~\cite{N2, CNS, N1}, the solvability for Dirichlet, regularity and Neumann problems with data in $\L^2$ was established for  parabolic equations as \eqref{eq1l} under the additional assumptions that the elliptic part is independent of the time variable $t$ and that it has either constant (complex) coefficients, real symmetric coefficients, or small perturbations thereof. The analysis in \cite{N2, CNS, N1} was advanced further in~\cite{AEN}, where a first order strategy to study boundary value problems for parabolic systems with second order elliptic part in the upper half-space was developed. The outcome of~\cite{AEN} was the possibility to address arbitrary parabolic equations (and systems) as in \eqref{eq1l} with coefficients depending also on time and  the transverse variable with additional transversal regularity. Also, in \cite{AEN1} parabolic equations as in \eqref{eq1l} were considered, assuming that the coefficients are real, bounded, measurable, uniformly elliptic, but not necessarily symmetric, and  the solvability of the Dirichlet problem with data
    in $\L^p$ was established.

    A particular outcome of the technology developed in \cite{AEN} was the resolution of a parabolic version of the famous Kato square root conjecture, see also \cite{N1} for important preliminary work on the parabolic  Kato square root problem,  originally posed for second order elliptic operators by Tosio Kato and solved in the elliptic setting in \cite{AHLeT},  \cite{AHLMcT}. The  {maximal accretivity} of the operator $(\cH,\dom)$  and the resolution of the parabolic  Kato square root problem are fundamental to our work.

    Our contribution is twofold. First, by connecting several lines of thought in the literature we are able to establish the connection
    between $(\partial_t  -\div_{x} ( A(x,t)\nabla_{x}))^s$ and an extension problem as the one in  \eqref{DN2+} and \eqref{DN1+}, but with
    $( \partial_t- \Delta_x )$ replaced by $\cH=(\partial_t  -\div_{x} ( A(x,t)\nabla_{x}))$, see Theorem \ref{thm1}. Second, connecting the extension to (local) linear degenerate parabolic equations, we prove, assuming in addition that $A=A(x,t)=\{A_{i,j}(x,t)\}_{i,j=1}^{n}$ is real, but not necessarily symmetric, that solutions to $(\partial_t  -\div_{x} ( A(x,t)\nabla_{x}))^su=0$ are H{\"o}lder continuous, and that non-negative solutions
    satisfy the (classical) Harnack inequality for linear parabolic equations, see Theorem \ref{Holder} and Theorem \ref{Harnack}. The constants appearing in these results/estimates only depend on
    dimension $n$, and the boundedness and ellipticity of $A$.

\subsection{Organization of the paper} The rest of the paper is organized as follows.  In Section \ref{prelim} we introduce the operator $\cH^s$, $s\in (0,1)$, which formally coincides with
     $(\partial_t  -\div_{x} ( A(x,t)\nabla_{x}))^s$. We here also state the Kato square root estimate, its implication on explicit descriptions of $\dom({\cH}^s)$, we recall facts from semigroup theory used in the paper, and we define what we mean by a solution to $\cH^su=0$ in $\Omega\times J$. In Section \ref{main} we formulate the extension problem and we state three results: Theorem \ref{thm1}, Theorem \ref{Holder} and Theorem \ref{Harnack}. In this section we also briefly discuss the path to the proofs of our results and we relate our effort concerning the extension problem to the vast operator theoretical literature on the topic. In Section \ref{sec2} we prove Theorem \ref{thm1}.  In Section
\ref{sec3}  we prove that local properties of solutions to the non-local Dirichlet problem associated to $(\partial_t  -\div_{x} ( A(x,t)\nabla_{x}))^s$ can be studied through
the corresponding problems  for local linear degenerate parabolic equations. In Section \ref{sec4} we specialize to real coefficients and prove Theorem \ref{Holder} and Theorem \ref{Harnack} using specific non-negative kernels which we derive based on the fundamental solutions for $\partial_t  -\div_{x} ( A(x,t)\nabla_{x})$ constructed in \cite{A}. In Section \ref{conc} we give some concluding remarks.

    \section{Parabolic operators, their powers and the Kato square root estimate}\label{prelim}

    Let $H:=\L^2(\mathbb R^{n+1}):=\L^2(\mathbb R^{n+1},\d x\d t)$ be the standard ${ \L}^2$ space of complex valued functions on $\mathbb R^{n+1}$ equipped with inner product $\langle\cdot,\cdot\rangle:=\langle\cdot,\cdot\rangle_H$ and norm $||u||_{2}:=||u||_{H}:=\langle u,u\rangle^{1/2}$. We introduce the  {energy space} $\dot{\E}(\ree)$ by taking the closure of all (complex) test functions $v \in \C_0^\infty(\ree)$ with respect to
\begin{align*}
\|v\|_{\dot{\E}(\ree)} := ||\gradx v||_{2}+||\dhalf v||_{2}.
\end{align*}
The half-order $t$-derivative $\dhalf$ is defined via the Fourier symbol $ |\tau|^{1/2}$.  The corresponding inhomogeneous energy space $\E(\ree) := \dot{\E}(\ree) \cap \L^2(\ree)$ is equipped with the  Hilbertian norm
\begin{align*}
\|v\|_{{\E}(\ree)} := \bigl (||v||_{2}^2+ ||\gradx v||_{2}^2+||\dhalf v||_{2}^2\bigr )^{1/2}.
\end{align*}
For short we have the triple
\begin{align}\label{spaces}
H=\L^2(\mathbb R^{n+1}),\ V:={\E}(\ree),\ V':={\E}(\ree)^*,
\end{align}
where $V'={\E}(\ree)^*$ is the (anti)-dual of $V={\E}(\ree)$. Hence we consider  two Hilbert spaces $H$ and  $V$ such that
$$V\hookrightarrow H,$$
i.e., $V$  is continuously and densely embedded in $H$, and $$V\hookrightarrow H\hookrightarrow V'.$$

    Let $A=A(x,t)=\{A_{i,j}(x,t)\}_{i,j=1}^{n}$ be a complex, measurable, $n\times n$-dimensional matrix valued function such that
 \begin{equation}\label{ellip}
 c_1|\xi|^2\leq \Re (A(x,t) \xi \cdot \cl{\xi}), \qquad
 |A(x,t)\xi\cdot\zeta|\leq c_2 |\xi||\zeta|,
\end{equation}
for some $c_1$, $c_2\in (0,\infty)$, and for all $\xi,\zeta\in \mathbb C^{n}$, $(x,t)\in\mathbb R^{n+1}$. Based on $A$ we introduce the sesquilinear form
\begin{align*}
\mathcal E(u,v) := \iint_{\ree} A(x,t) \gradx u \cdot \cl{\gradx v} +\HT\dhalf u \cdot \cl{\dhalf v} \, \d x\d t,\ u,v\in V,
\end{align*}
where  $\HT$ denotes the Hilbert transform with respect to the $t$-variable. The sesquilinear form induces a bounded operator $\cH$ from $V$ into $V'$ via
\begin{align}
\label{hidden coercivity}
\langle \cH u,v\rangle_{V',V} := \mathcal E(u,v),\ u,v\in V,
\end{align}
i.e., $\cH\in\mathcal L(V,V')$, where $\mathcal L(V,V')$ is the space of all linear
and bounded operators from $V$ to $V'$.

While $\cH$ initially is an unbounded operator on $H$ we consider its restriction to
\begin{align}\label{domain}
\dom:= \{u \in V : \cH u \in H \}.
\end{align}
Recall that by definition this means that if $u \in V$, then
$u\in\dom$ if and only if there exists a constant c such that,
\begin{align*}
|\mathcal E(u,v)|= \biggl |\iint_{\ree} A(x,t) \gradx u \cdot \cl{\gradx v} +\HT\dhalf u \cdot \cl{\dhalf v} \, \d x\d t\biggr |\leq c ||v||_2,
\end{align*}
for all $v \in V$. Note that boundary conditions  are encoded in $\dom$ by a formal integration by parts only if one restricts to the part of $\cH$ in $H$. Note also that formally the sesquilinear form induces, if we factorize $\partial_t = \dhalf \HT \dhalf$, the second order parabolic operator
 \begin{eqnarray}\label{eq1deg-}
\partial_t  -\div_{x} ( A(x,t)\nabla_{x}),\ (x,t)\in \mathbb R^{n+1}.
 \end{eqnarray}
Throughout the paper we will, unless otherwise stated, identify $\cH$ with its restriction to the domain $\dom$ introduced in \eqref{domain}.

 \subsection{Maximal accretivity and the definition of $\cH^s$} Recall that an operator $\cH$ in $H$ is {maximal accretive} if $\cH$ is closed and for every $\sigma \in \IC$ with $\Re \sigma <0$, the operator $\sigma - {\cH}$ is invertible on $H$ and the resolvent
$(\sigma - {\cH})^{-1}$ satisfies the estimate $\|(\sigma - {\cH})^{-1}\|_{H \to H} \leq (|\Re \sigma|)^{-1}$. The starting point for this paper is the following theorem.
 \begin{thm}\label{Maxdom} The part of $\cH$ in $H$, with maximal domain $\dom$ defined in \eqref{domain},
is {maximal accretive}. The analogous result holds for the dual of $\cH$, $\cH^\ast$.
\end{thm}

 The proof of Theorem \ref{Maxdom} can be found in Lemma 4 in \cite{AE}.
Using Theorem \ref{Maxdom} the fractional powers
 ${\cH}^s$, for $s\in (0,1)$, are well-defined and we will connect them to a local extension problem. By Theorem \ref{Maxdom}, the operator ${\cH}$ is maximal accretive and therefore ${\cH}$  has a bounded $H^\infty$-calculus. Using this the fractional powers
 ${\cH}^s$, $s\in (0,1)$,  are well-defined through the  Balakrishnan representation
 \begin{align}\label{ba}
 {\cH}^su:=\frac{\sin(s \pi)}{\pi}\int_0^\infty \lambda^{s-1}(\lambda+\cH)^{-1}\cH u\, \d \lambda,
\end{align}
for $u\in \dom$. For background on (maximal) accretive operators, dissipative operators, semigroup theory, sectorial operators, functional calculus, $H^\infty$-calculus and fractional operators, we refer in particular to \cite{Kato,Haase,Mc-,Mc,MS,Yosida}.

The domain of ${\cH}^s$, $\dom({\cH}^s)$, is the  space  $\{u\in H:\ {\cH}^su\in H\}$ equipped with the graph norm
\begin{align}\label{do1}\|u\|_{\dom({\cH}^s)}:=(\|u\|_2^2+\|\cH^s u\|_2^2)^{1/2}.
\end{align}
In particular, $\dom({\cH}^s)$ is a Hilbert space. If $0<s_1\leq s_2<1$, then
\begin{align}\label{do2}
\dom\subset\dom({\cH}^{s_2})\subseteq \dom({\cH}^{s_1})\subset H,
\end{align}
and
\begin{align}\label{do3}\mbox{{$\dom$ is a core for ${\cH}^s$} for all $s\in (0,1)$},
\end{align}
 i.e.,
\begin{equation}\label{do4}
\mbox{$\{(u,\cH^su):\ u\in \dom\}$ is dense in $\{(u,\cH^su):\ u\in \dom({\cH}^s))\}$ in the graph norm}.
\end{equation}
Furthermore, $(\cH,\dom)$ and $(\cH^\ast,\dom(\cH^\ast))$, $\dom(\cH^\ast):=\{u \in V : \cH^\ast u \in H \}$ are densely defined operators in $H$ in the sense that $\dom$ and $\dom(\cH^\ast)$  are dense in $H$.

We will need the fact that
\begin{align}\label{inter}
\dom({\cH}^s)=[H,\dom]_s
\end{align}
where $[\cdot,\cdot]_s$ denotes complex interpolation. To conclude \eqref{inter} we first note that ${\cH}$ is one-to-one on $\dom$. Indeed, if ${\cH}u=0$, then
\begin{align}\label{relaa}
\mathcal E(u,v)=\iint_{\ree} A(x,t) \gradx u \cdot \cl{\gradx v} +\HT\dhalf u \cdot \cl{\dhalf v} \, \d x\d t=0,
\end{align}
for all $v\in V$.  Consider the modified sesquilinear form
\begin{align*}
 \mathcal E_\delta(u,v) := \iint_{\R^{n+1}} A(x,t)\nabla_x u \cdot \cl{\nabla_x(1+\delta \HT) v} + \HT \dhalf u \cdot \cl{\dhalf (1+\delta \HT) v}\, \d x\d t,
\end{align*}
where $\delta$ is a  real number yet to be chosen. The Hilbert transform $\HT$ is a skew-symmetric isometric operator with inverse $-\HT$ on $\dot{\E}(\mathbb R^{n+1})$ and $V$. Hence, $1+\delta \HT$ is invertible on these spaces for any $\delta \in \R$. The key observation  is that if $A$ satisfies  \eqref{ellip}, and if we fix $\delta>0$ small enough and only depending on the structural parameters, then $\mathcal E_\delta$ is a bounded coercive sesquilinear form on $\dot{\E}(\mathbb R^{n+1})$. Indeed, using \eqref{ellip}, we first have
\begin{align*}
| \mathcal E_\delta(u,v)|&\lesssim \|\nabla_x u\|_{2}\|\nabla_x v\|_{2}+\| \dhalf u\|_{2}\| \dhalf v\|_{2}.
\end{align*}
Second, following the same argument as \cite{N1}, we see that
\begin{align}\label{eq:coer+}
\Re \mathcal E_\delta(u,u)&\ge (c_1-c_2\delta)\|\nabla_x u\|_{2}^2 + \delta \|\HT \dhalf u \|_{2}^2.
\end{align}
In particular, choosing $\delta$ small enough, and just depending on the structural constants, we see that
\begin{align}\label{eq:coer+a}
\Re \mathcal E_\delta(u,u)&\gtrsim \|\nabla_x u\|_{2}^2 + \delta \|\HT \dhalf u \|_{2}^2.
\end{align}
Using \eqref{relaa} with $v=(1+\delta \HT) u$ we see that $\mathcal E_\delta(u,u)=0$ and hence by \eqref{eq:coer+a},
\begin{align}\label{eq:coer+abla}
\|\nabla_x u\|_{2}^2 + \delta \|\HT \dhalf u \|_{2}^2=0.
\end{align}
We can conclude that $u$ is constant, see Lemma 3.3 in \cite{AEN}. As $u\in\dom$ it follows that $u\equiv 0$. Having concluded that the maximal accretive operator $(\cH,\dom)$ is one-to-one on $\dom$, \eqref{inter} now follows from \cite{AMN}, see Section 5 in \cite{AMN}, or Corollary 4.30 in \cite{Lu}.

\subsection{The parabolic Kato square root estimate} The following theorem is the resolution of the parabolic version of the famous Kato square root conjecture proved in \cite{AEN}.
\begin{thm}\label{thm:Kato} The square root of $\cH$, $\sqrt{\cH}$, is well-defined and the domain of the square root is that of the accretive form, that is, $\dom(\sqrt{\cH}) = V$. The two-sided estimate
\begin{align*}
\|\sqrt {\cH}\, u\|_{2} \sim \|\nabla_x u\|_{2}+ \| \dhalf u\|_{2}  \qquad (u \in V),
\end{align*}
holds with implicit constants depending only upon n and ellipticity constants of $A$. The same results holds for the dual of $\cH$, $\cH^\ast$.
\end{thm}

\subsection{The domains of $\cH$ and ${\cH}^s$} To understand $\dom=\dom(\cH)$  is a largely open problem often referred to as the maximal regularity problem:  to prove that solutions have a full time derivative in $H$, see \cite{Lions-Problem}. However, this seems to require regularity of the coefficients in $t$ at the order of a half time derivative, more precisely, this is what proofs require but, strictly speaking, and to the knowledge of the authors, there are no counterexamples showing that it is really necessary. In the case when $A$ is independent of $t$, then $\dhalf$ and $\div_{x} ( A(x)\nabla_{x})$ commute. Using this, we consider $u\in\dom$ and we let $f:=\cH u\in H$. Arguing formally we see that $\dhalf f\in V'$, and using the idea of hidden coercivity discussed above, and Cauchy-Schwarz with $\epsilon$, we can conclude that
\begin{align}\label{estacor}
 \|\nabla_x \dhalf u\|_{2}+ \| \dhalf \dhalf u\|_{2}\lesssim \|f\|_2.
\end{align}
In particular, $\partial_tu\in H$ and $u\in \L^2(\mathbb R,\dom(\div_{x} ( A(x)\nabla_{x})))$. In fact, we obtain that
\begin{align}\label{estacorll}
 \|\cH u\|_{2}\sim \| \partial_tu\|_{2}+\|\div_{x} ( A(x)\nabla_{x}u))\|_2,
\end{align}
whenever $u\in\dom$. This argument can be made rigorous by considering a
regularization of $f$, $f_\delta$, such that $f_\delta\to f$ in $H$ as $\delta\to 0$, and by considering $u_\delta$ such that $\cH u_\delta=f_\delta$. Then
\eqref{estacor} remains true with $(u,f)$ replaced by $(u_\delta, f_\delta)$ and the conclusion follows by taking limits as $\delta\to 0$. We omit further details.

To understand $\dom(\cH^s)$, one can use Theorem \ref{thm:Kato} to shed some light on $\dom({\cH}^s)$. By Theorem \ref{thm:Kato} we have
\begin{equation}\label{do2}
\dom({\cH}^{1/2})=V\mbox{ and }\|u\|_{\dom({\cH}^{1/2})}\sim\|u\|_2+\|\nabla_x u\|_{2}+ \| \dhalf u\|_{2}.
\end{equation}
Given $s\in (0,1)$ we introduce the parabolic Sobolev space $\H^{s}_{\pd_{t} - \Delta_x}$  defined as all functions $u\in H$ such that  $\|\cF^{-1}((|\xi|^2 + \i \tau)^{s/2} \cF u)\|_2<\infty$ where $\cF$ denotes the Fourier transform in the $(x,t)$ variables. We equip $\H^{s}:=\H^{s}_{\pd_{t} - \Delta_x}$ with the norm
 $$\|u\|_{\H^{s}}:=\bigl (\|u\|_2^2+\|\cF^{-1}((|\xi|^2 + \i \tau)^{s/2} \cF u)\|_2^2\bigr )^{1/2},$$
 and we note that $\H^{s}$ is a Hilbert space and that $\H^0=H$. Then, using Theorem \ref{thm:Kato} and interpolation, one can conclude that if $s\in (0,1/2]$, then
 $\dom(\cH^s)=\H^{2s}$ and
 \begin{align*}
\|\cH^s\, u\|_{2} \sim \|\cF^{-1}((|\xi|^2 + \i \tau)^{s} \cF u)\|_2  \qquad (u \in \H^{2s}),
\end{align*}
where now the implicit constants also depend on $s$. While this gives an explicit description of $\dom({\cH}^s)$ for $s\in (0,1/2]$, the situation is less clear for
$s\in (1/2,1)$.  Indeed, given $s\in (1/2,1)$ and writing $\cH^s=\cH^{1/2}\cH^{s-1/2}$, we see that
$\dom(\cH^s)=\{u\in H:\ \cH^{s-1/2}u\in \H^1\}$. To further understand $\dom(\cH^s)$ for $s\in (1/2,1]$ is more complicated though as the case $s=1$ is the maximal regularity problem discussed above.

\subsection{Semigroup theory}\label{sec2-}

As $(\cH,\dom)$ is maximal accretive, $(-\cH,\dom)$ is maximal dissipative, in $H$. Therefore, using the Hille-Yosida or Lumer-Phillips theorem in Hilbert spaces,
$-\cH$ is the infinitesimal generator of a strongly continuous semigroup of contractions, $\mathcal{S}=\mathcal{S}(\lambda)$,  on $H$. In particular, there exists a mapping $\mathcal{S}:[0,\infty)\to \mathcal{L}(H,H)$ such that
\begin{align*}
(i)&\quad\mathcal{S}(0)=I,\\
(ii)&\quad\mathcal{S}(\lambda_1+\lambda_2)=\mathcal{S}(\lambda_1)\mathcal{S}(\lambda_2),\mbox{ for all $\lambda_1,\lambda_2\in (0,\infty)$},\\
(iii)&\quad \lim_{\lambda\downarrow 0}||\mathcal{S}(\lambda)u-u||_2=0\mbox{ for all $u\in H$},
\end{align*}
such that\begin{align}\label{contractive}
 ||\mathcal{S}(\lambda)||_{H\rightarrow H}\leq 1, \mbox{ for all $\lambda\in (0,\infty)$},
\end{align}
and such that
$$-\cH u=\lim_{\lambda\downarrow 0}\frac {(S(\lambda)-I)u}\lambda,$$
whenever $u\in\dom$.

Given $u\in \dom$, $\tilde u(\lambda):=\mathcal{S}(\lambda)u$ is the  unique strong solution to the problem
\begin{align}
 &\tilde u\in C^0([0,\infty),\dom)\cap C^1([0,\infty),H),\notag\\
&\tilde u'(\lambda) +\cH \tilde u(\lambda)=0,\mbox{ for } \lambda\in (0, \infty),\ \tilde u(0)=u.
 \label{Semig}
\end{align}
If $u\in \dom$, then $\tilde u(\lambda)=\mathcal{S}(\lambda)u\in\dom$ and
\begin{align}
\label{id1}\tilde u'(\lambda)=-\mathcal{S}(\lambda)\cH u.
\end{align}
Note that the equation in \eqref{Semig} can{, and should,} be interpreted as,
\begin{align}\label{42semi}
\langle \tilde u'(\lambda),v\rangle_H =-\langle \cH\tilde u(\lambda),v\rangle_{V',V}=-\mathcal{E}(\tilde u(\lambda),v) \mbox{ for all $v\in V$, $\lambda \in (0,\infty)$},
\end{align}
and
\begin{align}\label{42+semi}
\lim_{\lambda\downarrow 0}\tilde u(\lambda)=u\mbox{ in $H$}.
\end{align}
Furthermore, if
$u\in \dom$, and $0\leq\lambda_2\leq \lambda_1$, then
\begin{align}
\label{id2}
\mathcal{S}(\lambda_1)u-\mathcal{S}(\lambda_2)u=-\int_{\lambda_2}^{\lambda_1}\mathcal{S}(\lambda)\cH u\, \d\lambda=-\int_{\lambda_2}^{\lambda_1}\cH\mathcal{S}(\lambda) u\, \d\lambda.
\end{align}

Note that we have
\begin{align}
\label{importantest-}
||(\mathcal{S}(\lambda)-I)u||_2&\leq 2||u||_2=(||u||_2+||\mathcal{H}^0u||_2),
\end{align}
for all $u\in H$ where we  identify $\mathcal{H}^0$ with the identity operator. Using \eqref{id2} and \eqref{contractive} we see that
\begin{align}
\label{importantest}
||(\mathcal{S}(\lambda)-I)u||_2&\leq \int_{0}^{\lambda}||\mathcal{S}(r)\cH u||_2\, \d r\leq \lambda||\cH u||_2{\leq \lambda(||u||_2+||\cH u||_2)} <\infty,
\end{align}
for all $u\in \dom$. Let $T_\lambda:=(\mathcal{S}(\lambda)-I)$ and consider the (function space) couples {$(H,\dom)$ and $(H,H)$.} Using \eqref{inter} and complex interpolation, see for example
Theorem 2.6 in \cite{Lu}, we deduce that

\begin{equation}\label{importantest+-}
{||T_\lambda||_{\dom(\mathcal{H}^s) \to H} \leq ||T_\lambda||_{\dom \to  H}^{(1-s)}||T_\lambda||_{H\to H}^s\leq c \lambda^s,}
\end{equation}
for all $s\in (0,1)$. In particular,
\begin{align}\label{importantest+}
||(\mathcal{S}(\lambda)-I)u||_2&\leq {c \lambda^s(||u||_2+||\cH^s u||_2)}<\infty,
\end{align}
for all $u\in \dom(\mathcal{H}^s)$.

We also note, see Proposition 3.2.1 in \cite{MS}, that we can use $\mathcal{S}$ to express ${\cH}^s$ as
\begin{align}\label{semiuu}
 {\cH}^su=\frac{1}{\Gamma(-s)}\int_0^\infty \lambda^{-s-1}(\mathcal{S}(\lambda)-I)u\, \d \lambda
\end{align}
for $u\in \dom${, and using Corollary 5.1.12 in \cite{MS}, we may extend \eqref{semiuu} to hold for $u\in\dom(\mathcal{H}^s)$}. 
Finally, let
\begin{align}\label{hille-}
R_{m}(\lambda):=\biggl(I+\frac{\lambda}{m}\cH\biggr)^{-m}.
\end{align}
Then the $C_0$-semigroup $\mathcal{S}(\lambda)$ generated by $-\cH$ can be identified,  following the proof of Hille, as
\begin{align}\label{hille}
\mathcal{S}(\lambda)u=\lim\limits_{m\to \infty} R_{m}(\lambda)u
\end{align}
for $\lambda>0$ and  for all $u \in H$, see \cite[Section~IX.1.2]{Kato}. In this sense we can formally state that $\mathcal{S}(\lambda)u=e^{-\lambda\cH}u$.


\subsection{Definition of solutions to $\cH^su=0$ in $\Omega\times J$} Let $\Omega\subset\mathbb R^n$ be a domain, and let $J\subset\mathbb R$ be an interval. We consider solutions to the non-local Dirichlet problem
\begin{align}\label{DP}
\cH^su=(\partial_t  -\div_{x} ( A(x,t)\nabla_{x}))^su&=0\mbox{ for $(x,t)\in \Omega  \times J$},\notag\\
u&=f\mbox{ for $(x,t)\in \mathbb R^{n+1}\setminus (\Omega \times J)$},
\end{align}
where $f:\mathbb R^{n+1}\setminus (\Omega \times J)\to\mathbb R$ is a given function.

 \begin{defn}\label{solu} We say that $u\in\dom (\cH^s)$ is a solution to $\cH^su=0$ in $\Omega\times J$ if $\langle \cH^su,\phi\rangle_H=0$ for all $\phi\in C_0^\infty(\Omega\times J)$.  Given $f\in H$ we say that $u\in\dom (\cH^s)$ is a solution to
the non-local Dirichlet problem in \eqref{DP}, if $u$ is a solution to $\cH^su=0$ in $\Omega\times J$  and if $u=f$ on $\mathbb R^{n+1}\setminus (\Omega\times J)$ in the sense that
$\langle (u-f),\phi\rangle_H=0$ for all $\phi\in C_0^\infty(\mathbb R^{n+1}\setminus (\Omega\times J))$.
\end{defn}


\section{Statement of our results}\label{main}

{Given $s\in (0,1)$ and  $\varepsilon\in (0,1)$  small,  $s+\varepsilon\leq 1$, let $\bar{s}=\max\{1/2,s+\varepsilon\}$.
We say that $\ru(\lambda)$ is a solution to}
\begin{equation}
(\lambda^{1-2s}\ru')'(\lambda) =\lambda^{1-2s}\cH \ru(\lambda), \quad  \lambda\in (0, \infty),\ \ru(0)  = u,
\label{ODE}
\end{equation}
if the following hold.
First,
{
\begin{align}\label{42-}
\ru(\cdot) &\in C_b^0\bigl( [0, \infty), \dom(\mathcal{H}^{\bar{s}}) \bigr) \cap C^{\infty}\bigl( (0, \infty), \dom(\mathcal{H}^{\bar{s}})  \bigr).
\end{align}}
Second,
\begin{align}\label{42}
\langle(\lambda^{1-2s}\ru')'(\lambda),v\rangle_H =\lambda^{1-2s}\langle \cH \ru(\lambda),v\rangle_{V',V} \mbox{ for all $v\in V$, $\lambda \in (0,\infty)$}.
\end{align}
Third,
\begin{align}\label{42+}
\lim_{\lambda\downarrow 0}\ru(\lambda)=u\mbox{ in $H$}.
\end{align}

We first prove the following theorem concerning the connection between ${\cH}^s$ and the extension problem \eqref{42-}-\eqref{42+}.
{
\begin{thm}\label{thm1} Given $s\in (0,1)$ and  $\varepsilon\in (0,1)$  small,  $s+\varepsilon\leq 1$, let $\bar{s}=\max\{1/2,s+\varepsilon\}$. Define $\mathcal U\colon[0,\infty)\to\mathcal L(\dom(\mathcal{H}^{\bar{s}}))$ as
\begin{align}\label{51}
\mathcal U(\lambda):=\frac{1}{\Gamma(s)}\int_0^\infty  h_s({\lambda^2}/{4r})\mathcal{S}(r) \, \frac{\d r}{r}.
\end{align}
Let $\ru (\lambda):=\mathcal U(\lambda)u$, $u\in \dom(\mathcal{H}^{\bar{s}})$.  Then $\ru (\lambda)$ is a solution to \eqref{ODE} in the sense of \eqref{42-}-\eqref{42+} and $\ru (\lambda)$  satisfies
\begin{align}\label{53+---}
\lim_{\lambda\to\infty}\langle \ru (\lambda),v\rangle_H=0,\mbox{ for all $v\in H$.}
\end{align}
Furthermore,
\begin{align}\label{53+--}
||\lambda^{1-2s}\ru'(\lambda)||_2\leq c\max\{1,|\lambda|^{2\varepsilon}\}(||u||_{2}+|| \mathcal{H}^{s+\varepsilon}u||_{2}),
\end{align}
whenever $\lambda\in (0,\infty)$, where $c$ is independent of $u$ and $\lambda$ but depends on $s$ and $\varepsilon$. Also,
\begin{align}\label{53+-}
-\lim_{\lambda\downarrow 0}\lambda^{1-2s}\ru'(\lambda)=-\lim_{\lambda\downarrow 0}\lambda^{1-2s}\frac {(\ru (\lambda)-\ru (0))}\lambda=c_s{\cH}^su\mbox{ in $H$,\  $c_s:=2^{1-2s}\frac{\Gamma(1-s)}{\Gamma(s)}$.}
\end{align}
\end{thm}}

\begin{rem}\label{remma1} It is important to note that to have $||\lambda^{1-2s}\ru'(\lambda)||_2$ finite in \eqref{53+--} it is sufficient to assume that {$u\in \dom(\mathcal{H}^{\bar s})$}, which is weaker than $u\in \dom$, but still stronger than $u\in \dom(\mathcal{H}^{s})$. In \eqref{53+--} the constant $c$ tends to $\infty$ as $\varepsilon\to 0$.
\end{rem}

Given $(x,t), (y,s)\in\mathbb R^{n+1}$, {and} $r>0$, we let $$d (x,t,y,s):=|x-y|+|t-s|^{1/2},$$
and
\[ Q_r ( x, t ) \, := \, \{ ( y, s )\in \mathbb R^{ n + 1 } : | y_i  - x_i| < r,\ t-r^2<s<t \}.\]
Note that by definition, $Q_r ( x, t )$ only contains points which are in the history relative $t$. Making  the additional assumption that $A=A(x,t)=\{A_{i,j}(x,t)\}_{i,j=1}^{n}$ is real and measurable, we derive the following local properties of solutions to the non-local Dirichlet problem in \eqref{DP}.

\begin{thm}\label{Holder} Assume that $A=A(x,t)=\{A_{i,j}(x,t)\}_{i,j=1}^{n}$ is real, measurable, and satisfies \eqref{ellip}. {Let $(z_0,\tau_0)\in\mathbb R^{n+1}$.} Given $s\in (0,1)$ and $\varepsilon\in (0,1)$ small,  $s+\varepsilon\leq 1$, let $\bar{s}=\max\{1/2,s+\varepsilon\}$. Assume that $u\in \dom({\cH}^{\bar{s}})$ is a solution to $\cH^s u=0$ in
${Q}_{4r}(z_0,\tau_0)$ and  that $$\|u\|^2_{{\L}^\infty(\mathbb R^n\times (-\infty,\tau_0])}<\infty.$$ Then, after a redefinition on a set of measure zero,
$u$ is continuous on ${Q}_{4r}(z_0,\tau_0)$. Furthermore, there exist constants $c$, $1\leq c<\infty$,  and $\alpha\in (0,1)$, both depending only on the structural constants $n$, $c_1$, $c_2$, and $s$ and $\bar{s}$,  such that
\[
|{u(x,t)-u(y,s)}|\le c\left(\frac{d (x,t,y,s)}{r}\right)^\alpha \|u\|_{{\L}^\infty(\mathbb R^n\times (-\infty,\tau_0])},
\]
whenever $(x,t)$, $(y,s)\in {Q}_{r}(z_0,\tau_0)$.
 \end{thm}

 \begin{thm}\label{Harnack} Assume that $A=A(x,t)=\{A_{i,j}(x,t)\}_{i,j=1}^{n}$ is real, measurable, and satisfies \eqref{ellip}. {Let $(z_0,\tau_0)\in\mathbb R^{n+1}$.} Given $s\in (0,1)$ and $\varepsilon\in (0,1)$ small,  $s+\varepsilon\leq 1$, let $\bar{s}=\max\{1/2,s+\varepsilon\}$. Assume that $u\in \dom({\cH}^{\bar{s}})$ is a solution to $\cH^s u=0$ in
${Q}_{4r}(z_0,\tau_0)$, that $$\|u\|^2_{{\L}^\infty(\mathbb R^n\times (-\infty,\tau_0])}<\infty,$$ and that $u\geq 0$ on $\mathbb R^n\times (-\infty,\tau_0]$.  Then there exist a constant $c$, $1\leq c<\infty$, depending only on the structural constants $n$, $c_1$, $c_2$, and  $s$ and $\bar{s}$, such that
\begin{align*}
\sup_{Q_{2r}^-(z_0,\tau_0)} u
\leq c \inf_{Q_{2r}^+(z_0,\tau_0)} u,
\end{align*}
where
\begin{align*}
Q_{2r}^-(z_0,\tau_0)&:={Q_{2r}(z_0,\tau_0)}\cap \{(x,t):\ \tau_0-3r^2/4<t< \tau_0-r^2/2\},\notag\\
Q_{2r}^+(z_0,\tau_0)&:={Q_{2r}(z_0,\tau_0)}\cap \{(x,t):\ \tau_0-r^2/4<t< \tau_0\}.
\end{align*}
\end{thm}

\begin{rem}\label{remma2} Note that in {Theorem \ref{thm1}, Theorem \ref{Holder}, and Theorem \ref{Harnack}} we assume that $u\in \dom({\cH}^{\bar{s}})\subset \dom({\cH}^{{s}})$, i.e., these theorems are established under an assumptions stronger than $u\in\dom({\cH}^{{s}})$. {The reason for this will become clear in Section \ref{sec2} and Section \ref{sec3}.}
\end{rem}

\subsection{Proofs} Theorem \ref{thm1} is a consequence of \cite{gale2013}, except for \eqref{53+--}. However, for us \eqref{53+--}  is crucial when we establish the connection
between the fractional powers, the extension problem and associated boundary value problem, 
{which is not discussed in \cite{gale2013}}, and in particular not for
parabolic operators  allowing for non-symmetric and time-dependent coefficients. For this reason, below we supply the proof of Theorem \ref{thm1} in full detail. Concerning the non-local Dirichlet problem introduced in \eqref{DP}, we will prove that this problem can be studied through the corresponding problem for local operators
\begin{eqnarray}\label{eq1deg}
\mathcal{L}:= \div_{X} ( w(X)B(X,t)\nabla_{X})-w(X)\partial_t,\ (X,t)\in \tilde\Omega\times J.
 \end{eqnarray}
Here $X:=(\lambda,x)=(x_0,x)\in\mathbb R^{n+1}$, $\tilde\Omega:=I\times \Omega \subset\mathbb R^{n+1}$, where $I\subset\mathbb R$ is an interval, and $w=w(X):\mathbb R^{n+1}\to\mathbb R$ is a function such that
\begin{equation}\label{A2}
 \sup_{B_r}\biggl (\barint_{B_r} w(X)\, \d X\biggr )\biggl (\barint_{B_r} \frac 1 {w(X)}\, \d X\biggr )\leq c_3,
\end{equation}
for some $c_3\in (0,\infty)$ and where $B_r\subset\mathbb R^{n+1}$ denotes a standard Euclidean ball. Note that \eqref{A2} states that
$w=w(X)$ belongs to the Muckenhoupt class $A_2(\mathbb R^{n+1},\d X)$. Here $B(X,t)=\{B_{i,j}(X,t)\}_{i,j=0}^{n}$ is a complex, measurable $(n+1)\times (n+1)$-dimensional matrix valued function such that
\begin{equation}\label{ellipaug}
 \kappa^{-1}|\xi|^2\leq  \Re (B(X,t) \xi \cdot \cl{\xi}), \qquad
 |B(X,t)\xi\cdot\zeta|\leq \kappa |\xi||\zeta|,
\end{equation}
for some  $\kappa\in [1,\infty)$, and for all $\xi,\zeta\in \mathbb C^{n+1}$, $(X,t)\in\mathbb R^{n+2}$. In the particular case of the non-local Dirichlet problem in \eqref{DP},
\begin{align}\label{Augmatrix}
B(X,t)=
\begin{bmatrix}
1 & 0 \\
0 & A(x,t)
\end{bmatrix},\ w(X)=w(x_0,x)=w(\lambda,x)=|\lambda|^{1-2s},
\end{align}
as $w(X)=w(x_0,x)=w(\lambda,x)=|\lambda|^{1-2s}\in A_2(\mathbb R^{n+1},\d X)$ whenever $s\in (0,1)$.

Assuming, in addition, that $A$, and hence $B$ in \eqref{Augmatrix}, is real, but not necessarily symmetric, Theorem \ref{Holder} and Theorem \ref{Harnack} states  local H{\"o}lder continuity for solutions to \eqref{DP}, and a Harnack inequality. To prove Theorem \ref{Holder} and Theorem \ref{Harnack} is not completely straightforward, one reason being a lack of $\L^\infty$ estimates for the semigroup $\mathcal{S}$ used in the extension problem. Indeed, assume that
$u\in\dom({\cH}^{\bar{s}})$ is a solution to $\cH^s u=0$ in
${Q}_{4r}(z_0,\tau_0)$ in the sense of Definition \ref{solu}. Then, as we will see, there exists a (traditional) weak solution
 $\tilde {\mathcal U}$ to the equation
\begin{eqnarray}\label{eq1deg+aka}
\mathcal{L}\,  \tilde {\mathcal U} = \div_{X} ( w(X)B(X,t)\nabla_{X} \tilde {\mathcal U} )-w(X)\partial_t \tilde {\mathcal U}  = 0
 \end{eqnarray}
 in ${Q}_{4r}(Z_0,\tau_0)$, $Z_0=(0,z_0)$, such that $\tilde {\mathcal U}(0,x,t)=u(x,t)$ for  a.e. $(x,t)\in {Q}_{4r}(z_0,\tau_0)$. $\tilde {\mathcal U}$ is
 defined in \eqref{globaldpakja} based on
 \begin{align}\label{51ll}
\mathcal U(X,t)=\mathcal U(\lambda,x,t)=\frac{1}{\Gamma(s)}\int_0^\infty  h_s({\lambda^2}/{4r})\mathcal{S}(r)u(x,t) \, \frac{\d r}{r}.
\end{align}
 The underlying idea is to derive regularity/estimates for $u$ based on corresponding estimates for  $\tilde {\mathcal U}$. Hence, one would like at least, to start with, to be able to control say the supremum of  $\tilde {\mathcal U}$ using $u$. Furthermore, it would be an advantage to know that if $u$ is globally positive, then so is $\tilde {\mathcal U}$. These considerations boil down to properties of the semigroup $\mathcal{S}$ and to the potential existence of a kernel representation of the semigroup, and properties of the kernel. Based on the generality of our setting the construction of the kernel for $\mathcal{S}$ is a rather difficult issue in our context. However, using \eqref{hille}, and Gaussian estimates for the fundamental solution for the operator $\cH$ established in ~\cite{A},  we are able  to derive, for our purposes, some approximating kernels and estimates thereof.

\subsection{Perspectives: the extension problem and operator theory} As we in this paper consider operators to which more traditional Fourier analytic, and spectral analysis, techniques do not seem to apply, and as we therefore have had to dive deeper into the world of operator theory,  we believe that it is relevant to give further background on the topics of this paper, and to put our efforts and in particular the extension result into context. Indeed, given \eqref{DN2}, \eqref{DN1} it is  natural to ask, for say a sectorial operator $A$ on a Banach space $X$, if one can define a {function space} valued ODE and a solution
$\ru$,
\begin{align}
 & (\lambda^{1-2s}\ru'(\lambda))'= \lambda^{1-2s}A\ru(\lambda), \quad  \lambda\in (0, \infty),\ \ru(0)  = u, \label{fractional_ODE}
\end{align}
where  $u\in \dom (A^s)$, such that
\begin{equation}
 - \lim\limits_{\lambda \rightarrow 0+} \lambda^{1-2s} \ru'(\lambda) = c_{s} A^{s}u.   \label{limit}
\end{equation}
 Note that in general \eqref{fractional_ODE} is, in line with \eqref{DN2} and \eqref{DN2+}, a linear ODE in the Banach space $X$ with initial datum $u \in X$ which degenerates for $\lambda=0$, unless $s = 1/2$, and which is incomplete since no initial condition for $\ru'$ is given. The {problems} that arise include existence and uniqueness for \eqref{fractional_ODE} and \eqref{limit},  properties {of} the generalized Dirichlet to Neumann map $$u\to - \lim\limits_{\lambda \rightarrow 0+} \lambda^{1-2s} \ru'(\lambda),$$ and the relation between this Dirichlet to Neumann map and $c_{s} A^{s}$.

 The problems defined by \eqref{fractional_ODE}  and \eqref{limit}  have recently been studied rather extensively in the operator theory community as well as in the PDEs community, see \cite{stinga2010,gale2013,arendt2016,MSe} and the references therein. The theory of semigroups, see \cite{Yosida,EngelNagel2000,phillips1952}, plays a fundamental and prominent role in the field as $\ru$ in \eqref{fractional_ODE} is frequently
constructed using the operator valued map
\begin{align}\label{51intro}
\lambda\to\frac{1}{\Gamma(s)}\int_0^\infty h_s({\lambda^2}/{4r})\mathcal{T}(r) \, \frac{\d r}{r}.
\end{align}
Here $\mathcal{T}$  is the  semigroup generated by $A$, assuming that it exists and can be constructed.  If for example $A$ is sectorial with angle less than $\pi/2$ on a Hilbert space, then $\mathcal{T}$ can be constructed as an analytic semigroup using the functional calculus. For more general operators, as considered in this paper, one can  hope to be in the
realm of (strongly) continuous semigroups (of contractions) or, more generally, in the realm of the integrated semigroups of Hieber and Neubrander \cite{Neubrander,Hieber-Annalen,Hieber-JMathAnal,Hieber-Forum}. We refer to \cite{gale2013} for more. A subtle but  important point is {to} decide to what function space the  data $u$ in \eqref{fractional_ODE} is  to belong.  Obviously $u$ must belong to the domain of $A^{s}$, $\dom(A^{s})$, to have \eqref{limit} well defined, but one option is to restrict $u$ to $\dom(A)$ which in the case of sectorial operators is contained in $\dom(A^{s})$.  {Another} problem is to give a clear cut description of $\dom(A^{s})$. From our perspective, beyond \cite{gale2013} we  think that \cite{arendt2016} and \cite{MSe}  are two particularly interesting contributions to the study of \eqref{fractional_ODE}.

In  \cite{arendt2016} (see also \cite{AtE}), W. Arendt et al. studied, motivated by the results of Caffarelli and Silvestre \cite{caffarelli2007}, the precise regularity properties of the Dirichlet problem and the Neumann problem in Hilbert spaces for the equation in \eqref{fractional_ODE}. In this context, the Dirichlet to Neumann map becomes an isomorphism between certain interpolation spaces, depending on the {setup}, and the part of this map which belongs to the underlying Hilbert space is exactly the fractional power. In \cite{arendt2016}, their operator $A$ is not only a sectorial operator, but is generated from a coercive form. Coercivity of the underlying form plays a central role in \cite{arendt2016} and this condition is only relaxed in the final section of the paper where instead  the weaker condition that the form is sectorial with vertex zero is imposed.

In \cite{MSe},  J. Meichsner et al.  construct, for a given densely defined sectorial  operator $A$ on a (general) Banach space $X$, a solution to the initial value problem in \eqref{fractional_ODE}, a solution which turns out to be holomorphic in some sector determined by the angle of sectoriality of $A$. This solution is proven to be the unique solution to the initial value problem and it is proven that if the Dirichlet to Neumann operator is constructed based on the solution, then this operator equals $c_sA^{s}$. The construction in \cite{MSe} is based on the fact that $\sqrt{A}$ is sectorial with angle of sectoriality less than $\frac{\pi}{2}$ and thus gives rise to a holomorphic $C_0$-semigroup as well as a rich functional calculus.

Conceptually the approaches in \cite{arendt2016} and \cite{MSe} are rooted in functional calculus but differ in the construction of the extensions. In \cite{arendt2016} the extension is constructed using the operator valued map
\begin{equation}\label{arendtextension}
\lambda\to \frac{1}{\Gamma(s)}\int_0^\infty e^{-\frac{\lambda^2}{4r}}r^s e^{-rA}A^s \, \frac{\d r}{r},
\end{equation}
while in \cite{MSe} the extension is constructed using
\begin{equation}\label{MSeextension}
  z\to \frac{z^{2s}}{2 \Gamma(2s)} \int\limits_0^\infty r^{s - \frac{1}{2}}
  e^{-z \sqrt{A + r}} (A + r)^{-\frac{1}{2}} \, \d r,\ z\in S_{(\pi-\omega)/2},
 \end{equation}
 if $A$ is sectorial with angle $\omega\in [0,\pi)$. To achieve boundedness estimates for the operator valued map in \eqref{arendtextension}, coercivity estimates for $A$ are relevant. In \eqref{MSeextension} it is important to note that $S_{\theta}$ is the open sector
 $\left\{ z \in \mathbb C \setminus (-\infty, 0]
	                     :\ \abs{\arg z } < \theta \right\}$. In particular, an analysis shows that the operator valued map in \eqref{MSeextension} is not bounded on $X$ at $z=0$ unless $A$ is bounded.

 Considering the weak assumptions on our operator $\cH$, $(\cH,\dom)$ is only maximal accretive and the underlying sesquilinear form is not directly coercive, and considering the fact that we want to work with weak solutions for the PDE defined through the extension, we in this paper construct extensions associated to $\cH^s$ using the semigroup approach in \eqref{51intro}.  To be able to apply the extension chosen to the study of local regularity for equations with real coefficients, we then have to derive appropriate kernel representations.  We think that it is an interesting problem to construct extensions associated to $\cH^s$ using the functional calculus approach in \eqref{arendtextension}, and to try to use the idea of hidden coercivity for parabolic operators, explored in  \cite{N2, CNS, N1, AEN, AEN1}, in the context of \cite{arendt2016}.  Concerning the functional calculus approach in \eqref{MSeextension}, the work in \cite{MSe} is directly applicable to the operator $\cH$ as $(\cH,\dom)$ is maximal accretive and hence the extension \eqref{MSeextension} is well-defined.  Theorem 4.9 and Theorem 5.8 in \cite{MSe} prove that the initial value problem in \eqref{fractional_ODE} has a unique solution for all $s\in (0,1)$, in fact even for all $s \in \mathbb C_{0< \Re s < 1}$, and all densely defined sectorial operators $A$ in a Banach space $X$. In our context their extension is
 $\mathcal{U}(z)=u_{z,s}(\sqrt{\cH})u$ where the function $u_{z,s}$ is defined in Definition 3.1 in \cite{MSe}. Furthermore, if $u\in\dom(\cH^s)$ it seems to follow from their approach that
 \begin{align*}
  -z^{1-2s} \partial_z u_{z, s} (\sqrt{\cH}) u & = c_{s} u_{z, s - 1}(\sqrt{\cH}) \cH^{s}u,
 \end{align*}
 and, in particular, restricting to $\lambda\in\mathbb R_+$,
 \begin{align*}
  \lim\limits_{{\lambda \to 0^+}} -\lambda^{1-2s} \partial_z u_{\lambda, s} (\sqrt{\cH})u = c_{s} \cH^{s}u.
  \end{align*}
  In many respects, \cite{MSe} gives a rather complete analysis of \eqref{fractional_ODE} and \eqref{limit} for sectorial operators. Still, as mentioned in Section 4 in \cite{MSe}, and which is clear from an analysis of \eqref{MSeextension}, the function $U$ lacks continuity at $z=0$ in the {norm-topology of $\mathcal{L}(X)$.} For us this is problematic, as we want to consider weak solutions for PDEs defined based {on} the extension, and  in particular we need \eqref{53+--}, and \eqref{space2reg+} stated below.

\section{The extension problem: proof of Theorem \ref{thm1}}\label{sec2}

{Given $s\in (0,1)$ and $\varepsilon\in (0,1)$ small,  $s+\varepsilon\leq 1$, let $\bar{s}=\max\{1/2,s+\varepsilon\}$. Let $\ru(\lambda):=\mathcal U(\lambda)u$, $u\in \dom(\mathcal{H}^{\bar{s}})$}, be as in the statement of the theorem. Note that by {the} definition of $\Gamma(s)$
\begin{align}\label{51a}
\frac{1}{\Gamma(s)}\int_0^\infty h_s({\lambda^2}/{4r}) \, \frac{\d r}{r}=1.
\end{align}
Applying \eqref{contractive} and \eqref{51a} we deduce
 \begin{align}\label{51apa}
  \frac{1}{\Gamma(s)}\int\limits_0^{\infty} \norm{h_s({\lambda^2}/{4r}) \mathcal{S}(r) f}_{2} \, \frac{\d r}{r} \leq ||f||_{2} \frac{1}{\Gamma(s)}\int\limits_0^{\infty} h_s({\lambda^2}/{4r}) \, \frac{\d r}{r}=||f||_{2},
 \end{align}
 whenever $f\in H$.
 {Hence
 \begin{equation}
     \begin{split}
        \norm{\mathcal{U}(\lambda)u}_2 = \frac{1}{\Gamma(s)} \int_0^\infty \norm{h_s(\lambda^2/4r) \mathcal{S}(r)u}_2 \frac{\d r}{r} \leq  \norm{u}_2 < \infty.
     \end{split}
 \end{equation}
 Using \eqref{hille}, together with the fact that by the functional calculus we have,
 \begin{equation*}
    \mathcal{H}^{\bar s}R_m(\lambda) = R_m(\lambda)\mathcal{H}^{\bar s}
 \end{equation*}
 on $\dom(\mathcal{H}^{\bar s})$, we can conclude that $\mathcal{H}^{\bar s}$ commutes with the semigroup $\mathcal{S}$ on $\dom(\mathcal{H}^{\bar s})$.
Using this together with \eqref{51apa}, we see that for all $u\in\dom(\mathcal{H}^{\bar s}u)$,
 \begin{equation*}
 \begin{split}
    \|\mathcal{H}^{\bar s}\mathcal{U}(\lambda)u\|_2 &= \frac{1}{\Gamma(s)} \int_0^\infty \norm{h_s(\lambda^2/4r) \mathcal{H}^{\bar s}\mathcal{S}(r)u}_2 \frac{\d r}{r}\\
    &= \frac{1}{\Gamma(s)} \int_0^\infty \norm{h_s(\lambda^2/4r) \mathcal{S}(r)\mathcal{H}^{\bar s}u}_2 \frac{\d r}{r}\\
    &\leq \|\mathcal{H}^{\bar s}u\| < \infty,
 \end{split}
 \end{equation*}
 and we can conclude that $\mathcal{U}(\lambda)u \in \dom(\mathcal{H}^{\bar{s}})$.
 In particular, the mapping $\U(\lambda):{\dom(\mathcal{H}^{\bar{s}})}\to {\dom(\mathcal{H}^{\bar{s}})}$ is well-defined for $\lambda > 0$, and so is $\ru (\lambda)$ in the statement of the theorem.
 }
 Now, substituting $\tau := \tfrac{\lambda^2}{4r}$ in the representation for $\ru(\lambda)$ yields
 \[
  \ru(\lambda) = \frac{1}{\Gamma(s)}\int_0^\infty  h_s({\lambda^2}/{4r})\mathcal{S}(r)u \, \frac{\d r}{r}=\frac{1}{\Gamma(s)} \int\limits_0^{\infty} h_s(\tau) \mathcal{S} ( {\lambda^2}/{4\tau}) u \, \frac{\d\tau}{\tau}.
 \]
 The boundedness of $\mathcal{S}$, its strong continuity, and by the dominated convergence theorem, we conclude that {$\ru \in C_b^0\bigl( [0, \infty), \dom(\mathcal{H}^{\bar{s}}) \bigr)$. In particular, $\ru(0)=u$.} Observe that the integrand in the definition of $\ru(\lambda)$, as well as the factor $\lambda^{2s}$, are smooth for ${\lambda > 0}$. Therefore, for every such $\lambda>0$ we can choose a compact interval $I$ with $\lambda \in I \subset (0, \infty)$ and again apply dominated convergence proving the smoothness of $\ru$. In particular, {
 $\ru \in C_b^0\bigl( [0, \infty), \dom(\mathcal{H}^{\bar{s}}) \bigr) \cap C^{\infty}\bigl( (0, \infty), \dom(\mathcal{H}^{\bar{s}})\bigr)$}
 and hence the derivatives $\ru'$ and $\ru''$ are well defined. This proves \eqref{42-} and \eqref{42+}.

Second we prove \eqref{42}.
{We begin by remarking that the right hand side of \eqref{42} is well defined, since it is to be interpreted in the sense of \eqref{42semi}, and ${\rm u}(\lambda)\in \dom(\mathcal{H}^{\bar s})$, and hence by \eqref{do2} and Theorem \ref{thm:Kato}, we have ${\rm u}(\lambda)\in V$.}
For $\lambda> 0$, direct calculations show that
 \begin{align*}
  \ru'(\lambda) &=\frac{1}{\Gamma(s)}\int_0^\infty  \left(\frac\lambda{2r}\right)h_s'({\lambda^2}/{4r})\mathcal{S}(r)u \, \frac{\d r}{r},
 \end{align*}
 and
 \begin{align*}
  \ru''(\lambda) &=\lambda^{-1}\ru'(\lambda)+\frac{1}{\Gamma(s)}\int_0^\infty  \left(\frac\lambda{2r}\right)^2h_s''({\lambda^2}/{4r})\mathcal{S}(r)u \, \frac{\d r}{r}.
 \end{align*}
Hence,
 \begin{align*}
  \ru''(\lambda) + \frac{1-2s}{\lambda}\ru'(\lambda) & =\frac{1}{\Gamma(s)}\int_0^\infty  \left(\left(\frac\lambda{2r}\right)^2h_s''({\lambda^2}/{4r})+\frac{(2-2s)}{\lambda}\left(\frac\lambda{2r}\right)h_s'({\lambda^2}/{4r})\right)\mathcal{S}(r)u \, \frac{\d r}{r}.
 \end{align*}
 Also
 $$\frac{\d}{\d r}(r^{-1}h_s({\lambda^2}/{4r}))=r^{-1}\left(\left(\frac\lambda{2r}\right)^2h_s''({\lambda^2}/{4r})+\frac{(2-2s)}{\lambda}\left(\frac\lambda{2r}\right)h_s'({\lambda^2}/{4r})\right),$$
 and therefore
  \begin{align*}
  \ru''(\lambda) + \frac{1-2s}{\lambda}\ru'(\lambda) & =\frac{1}{\Gamma(s)}\int_0^\infty \frac{\d}{\d r}(r^{-1}h_s({\lambda^2}/{4r}))\mathcal{S}(r)u \, {\d r}.
 \end{align*}
 Using integration by parts  we see that
 \begin{align*}
 \frac{1}{\Gamma(s)}\int_0^\infty \frac{\d}{\d r}(r^{-1}h_s({\lambda^2}/{4r}))\mathcal{S}(r)u \, {\d r}&=-\frac{1}{\Gamma(s)}\int_0^\infty h_s({\lambda^2}/{4r})\frac{\d}{\d r}\mathcal{S}(r)u \, \frac {\d r}{r}\notag\\
 &=\frac{1}{\Gamma(s)}\int_0^\infty h_s({\lambda^2}/{4r})\mathcal{H}\mathcal{S}(r)u \, \frac {\d r}{r}= \cH \ru(\lambda).
 \end{align*}
In particular,
\begin{align*}
(\lambda^{1-2s}\ru')'(\lambda) =\lambda^{1-2s}\cH \ru(\lambda)\mbox{ for all } \lambda \in (0,\infty),
\end{align*}
in the sense defined in \eqref{42}. Hence the proof of \eqref{42} is complete.

\eqref{53+---}-\eqref{53+-} remain to be proven. \eqref{53+---} follows immediately from \eqref{51apa} and dominated convergence. Concerning
\eqref{53+--}, we write out all expressions explicitly. Let $\tilde c_s:=\frac{1}{\Gamma(s)}2^{-2s}$.  For $\lambda> 0$
 we see that
 \begin{align*}
  \ru'(\lambda) &= \frac{2s}\lambda \ru(\lambda) - \tilde c_s\lambda^{2s}\int\limits_0^{\infty} \left( \frac{\lambda}{2r} \right) e^{-\frac{\lambda^2}{4r}} r^{-s-1}\mathcal{S}(r) u\, \d r.
 \end{align*}
Hence
  \begin{align*}
    -\lambda^{1-2s} \ru'(\lambda) = \tilde c_s \Bigl( -2s \int\limits_0^{\infty} r^{-s-1} e^{-\frac{\lambda^2}{4r}} {\mathcal{S}}(r) u \, \d r
								      + \frac{\lambda^2}{2} \int\limits_0^{\infty} r^{-s-2} e^{-\frac{\lambda^2}{4r}} {\mathcal{S}}(r) u \, \d r  \Bigr).
  \end{align*}
  Note that
  \[
   -2s \tilde c_s= {\frac{-s}{\Gamma(s)2^{2s-1}} = \frac{\Gamma(1-s)}{2^{2s-1} \Gamma(s) \Gamma(-s)}=\frac{c_s}{\Gamma(-s)},}
  \]
  {where $c_s$ is as in \eqref{53+-}}.
  Using this we write
  \begin{align*}
   - \lambda^{1-2s} \ru'(\lambda) = c_{s} & \biggl( \frac{1}{\Gamma(-s)} \int\limits_0^{\infty} r^{-s-1} e^{-\frac{\lambda^2}{4r}} \left( {\mathcal{S}(r)} -I \right)u\, \d r  \\
				      & - \frac{\lambda^2}{4s \Gamma(-s)} \int\limits_0^{\infty} r^{-s-2} e^{-\frac{\lambda^2}{4r}} {\mathcal{S}(r)} u \, \d r  + \frac{1}{\Gamma(-s)} \int\limits_0^{\infty} r^{-s-1} e^{-\frac{\lambda^2}{4r}} u \, \d r \biggr).
  \end{align*}
   Integration by parts, in the second term on the second line in the last display, gives
  \begin{align*}
    -&\frac{\lambda^2}{4s \Gamma(-s)} \int\limits_0^{\infty} r^{-s-2} e^{-\frac{\lambda^2}{4r}} \mathcal{S}(r) u \, \d r \; + \;
    \frac{1}{\Gamma(-s)} \int\limits_0^{\infty} r^{-s-1} e^{-\frac{\lambda^2}{4r}} u \, \d r \\
  &= -\frac{\lambda^2}{4s \Gamma(-s)} \int\limits_0^{\infty} r^{-s-2} e^{-\frac{\lambda^2}{4r}} \big( \mathcal{S}(r)-I\big)u\,  \d r,
  \end{align*}
  and hence
  \begin{align*}
   - \lambda^{1-2s} \ru'(\lambda) = B_1(\lambda)u+B_2(\lambda)u,
  \end{align*}
  where
  \begin{align*}
  B_1(\lambda)u&:= c_{s}   \frac{1}{\Gamma(-s)} \int\limits_0^{\infty} r^{-s-1} e^{-\frac{\lambda^2}{4r}} \left( {\mathcal{S}(r)} -I \right)u\, \d r,\notag   \\
B_2(\lambda)u&:= -c_{s} \frac{\lambda^2}{4s \Gamma(-s)} \int\limits_0^{\infty} r^{-s-2} e^{-\frac{\lambda^2}{4r}} \big( \mathcal{S}(r)-I\big)u\,  \d r.
  \end{align*}
   Using \eqref{importantest-} {and \eqref{importantest+},} and splitting the domain of integration in $B_1(\lambda)u$ into the intervals $(0,1]$ and $[1,\infty)$, we deduce
  \begin{equation*}
      \begin{split}
        || B_1(\lambda)u||_{2}  &\leq c_s\frac{1}{\Gamma(-s)}\left( \int_0^{1} r^{-s-1}e^{-\frac{\lambda^2}{4r}} r^{s+\varepsilon}(||u||_{2}+|| \mathcal{H}^{s+\varepsilon}u||_{2}) \d r +2
        \int_{1}^\infty r^{-s-1}e^{-\frac{\lambda^2}{4r}} || u ||_2 \d r \right)\\
        &\leq c_s\frac{1}{\Gamma(-s)}\left( \int_0^{1} r^{\varepsilon-1}e^{-\frac{\lambda^2}{4r}}(||u||_{2}+|| \mathcal{H}^{s+\varepsilon}u||_{2}) \d r +2
        \int_{1}^\infty r^{-s-1}e^{-\frac{\lambda^2}{4r}} || u ||_2 \d r \right)\\
        &\leq c_s\frac{1}{\Gamma(-s)}\left( \int_0^{1} r^{\varepsilon-1}(||u||_{2}+|| \mathcal{H}^{s+\varepsilon}u||_{2}) \d r +2
        \int_{1}^\infty r^{-s-1} || u ||_2 \d r \right).
      \end{split}
  \end{equation*}
  Hence,
  \begin{equation*}
      \begin{split}
        || B_1(\lambda)u||_{2}  &\leq {c_s\frac{1}{\Gamma(-s)}\left( (1/\varepsilon+2/s)||u||_{2}+|| \mathcal{H}^{s+\varepsilon}u||_{2}/\varepsilon\right).}
      \end{split}
  \end{equation*}
  Similarly, {using \eqref{importantest+}} and \eqref{51a} we deduce that
    \begin{align*}
 || B_2(\lambda)u||_{2}&\leq c_{s} \frac{\lambda^2}{4s \Gamma(-s)} \int\limits_0^{\infty} r^{-s-2} e^{-\frac{\lambda^2}{4r}} ||\big( \mathcal{S}(r)-I\big)u||_2 \,  \d r  \\
 &\leq c_{s} \frac{\lambda^2}{4s \Gamma(-s)} \int\limits_0^{\infty} r^{-s-2} e^{-\frac{\lambda^2}{4r}} r^{s+\varepsilon}(||u||_{2}+|| \mathcal{H}^{s+\varepsilon}u||_{2}) \,  \d r \\
 &= c_{s} \frac{\lambda^2}{4s \Gamma(-s)} \int\limits_0^{\infty} r^{\varepsilon-2} e^{-\frac{\lambda^2}{4r}}(||u||_{2}+|| \mathcal{H}^{s+\varepsilon}u||_{2}) \,  \d r.
  \end{align*}
  Hence,
   \begin{align*}
 || B_2(\lambda)u||_{2} &\leq c|\lambda|^{2\varepsilon}(||u||_{2}+|| \mathcal{H}^{s+\varepsilon}u||_{2})<\infty.
  \end{align*}
 This completes the proof of \eqref{53+--}. Using this we can conclude, by dominated convergence and using \eqref{semiuu}, that
  \begin{align}\label{blafir}
   - \lim_{\lambda\downarrow 0} \lambda^{1-2s} \ru'(\lambda)&=\frac{c_s}{\Gamma(-s)}\int_0^\infty r^{-s-1} (\mathcal{S}(r)-I)u \, {\d r}=c_s{\cH}^su,
  \end{align}
  in $H$. Furthermore, using \eqref{51a} we see that
\begin{align*}
\lambda^{1-2s}\frac {(\ru(\lambda)-\ru(0))}\lambda=\frac{1}{\Gamma(s)}\biggl (\frac 1 2\biggr )^{2s}\int_0^\infty r^{-s} e^{-\frac{\lambda^2}{4r}}(\mathcal{S}(r)-I)u \, \frac{\d r}{r}.
\end{align*}
Hence, also
\begin{align*}
-\lim_{\lambda\downarrow 0}\lambda^{1-2s}\frac {(\ru(\lambda)-\ru(0))}\lambda&=-\frac{1}{\Gamma(s)}\biggl (\frac 1 2\biggr )^{2s}\int_0^\infty r^{-s} (\mathcal{S}(r)-I)u \, \frac{\d r}{r}\\
&=\frac{c_s}{\Gamma(-s)}\int_0^\infty r^{-s-1} (\mathcal{S}(r)-I)u \, {\d r}=c_s{\cH}^su,
\end{align*}
in $H$. Combining this with \eqref{blafir} proves \eqref{53+-}. The proof of Theorem \ref{thm1} is complete.

\section{Reinforced weak solutions to the (local) extension problem} \label{sec3}

Given a domain $\Omega\subset\mathbb R^n$, and $J\subset\mathbb R$,  we let ${\L}^2(\Omega\times J)$ be the Hilbert space with norm
\begin{align*}
\|v\|_{{\L}^2(\Omega\times J)} := \biggl (\iint_{\Omega\times J} |v(x,t)|^2\, \d x\d t\biggr )^{1/2}.
\end{align*}
Recall that  ${\E}(\mathbb R^{n+1})={\E}(\mathbb R^n\times\mathbb R)$ is the Hilbert space on $\mathbb R^n\times \mathbb R$ with norm
\begin{align*}
\|v\|_{{\E}(\mathbb R^n\times\mathbb R)} := \bigl (\|v\|_{{\L}^2(\mathbb R^n\times\mathbb R)}^2+\|\gradx v\|_{{\L}^2(\mathbb R^n\times\mathbb R)}^2+\|\dhalf v\|_{{\L}^2(\mathbb R^n\times\mathbb R)}^2\bigr )^{1/2}.
\end{align*}
By the Kato square root estimate, see \eqref{do2}, we have
\begin{align}\label{Katoim}
\|v\|_{{\E}(\mathbb R^n\times\mathbb R)} \sim  \bigl (\|v\|_{{\L}^2(\mathbb R^n\times\mathbb R)}^2+\|\cH^{1/2}v\|_{{\L}^2(\mathbb R^n\times\mathbb R)}^2\bigr )^{1/2}.
\end{align}

We let, throughout the section,  $I\subset\mathbb R$ be a finite interval. We will consider spaces of  functions $u$ on $I$ with values in ${\E}(\mathbb R^n\times\mathbb R)$, $u:I\to {\E}(\mathbb R^n\times\mathbb R)$. Derivatives will be taken in the {distributional sense}, i.e. using the elements of the 
space $C_0^\infty(I)$ of all infinitely differentiable complex valued
functions with compact support as test functions.
Let $u,v\in { \L}_{\loc}^1(I,{\E}(\mathbb R^n\times\mathbb R))$.
We say that $v$ is the {weak derivative} of $u$ if
\begin{align*}
-\int_0^\infty\varphi'(\lambda)\overline{u(\lambda)}\, \d\lambda=\int_0^\infty\varphi(\lambda)\overline{v(\lambda)}\, \d\lambda,
\end{align*}
for all $\varphi\in C_0^\infty(I)$. In that case we write $u':=v$. The weak derivative is unique if it exists and for all
$u\in C^1(I,{\E}(\mathbb R^n\times\mathbb R))$ the weak and classical derivatives coincide.

Given an interval $I\subset\mathbb R$, and $0<s<1$, we introduce the space $W_{1-s}(I,{\E}(\mathbb R^n\times\mathbb R))$ as the Hilbert space of  functions $u:I\to {\E}(\mathbb R^n\times\mathbb R)$, such that $u':I\to \L^2(\mathbb R^n\times\mathbb R)$, and   with norm
\begin{align*}
\|u\|_{W_{1-s}(I,{\E}(\mathbb R^n\times\mathbb R))}&:=\Big(\int_{I}\left(\|u(\lambda)\|_{{\E}(\mathbb R^n\times\mathbb R)}^2+\| u'(\lambda)\|_{{\L}^2(\mathbb R^n\times\mathbb R)}^2\right)|\lambda|^{1-2s}\, \d\lambda\Big)^{\frac 12}\notag\\
&=\Big(\int_{I}\left(\|u(\lambda)\|_{{\E}(\mathbb R^n\times\mathbb R)}^2+\|\partial_\lambda u(\lambda)\|_{{\L}^2(\mathbb R^n\times\mathbb R)}^2\right)|\lambda|^{1-2s}\, \d\lambda\Big)^{\frac 12}.
\end{align*}
We let $C_0^\infty(I, {\E}(\mathbb R^n\times\mathbb R))$ be the set of $\phi:I\to {\E}(\mathbb R^n\times\mathbb R)$ such that $\phi$ is $C^\infty(I)$ smooth with compact support on $I$.

  Let {$\mathcal U\colon[0,\infty)\to\mathcal L(\dom(\mathcal{H}^{\bar s}))$} be defined as in Theorem \ref{thm1} and let
$$\mathcal U(\lambda,x,t):=\mathcal U(\lambda)u(x,t),\ (\lambda,x,t)\in \mathbb R_+\times \mathbb R^n\times\mathbb R,\ {u\in \dom(\mathcal{H}^{\bar s}).}$$
Based on $ {\mathcal U}(\lambda,x,t)$ as above, we introduce $\tilde {\mathcal U}$ on $\mathbb R\times\mathbb R^n\times \mathbb R$ through
\begin{align}\label{globaldpakja}
\tilde {\mathcal U}(\lambda,x,t)&:= {\mathcal U}(\lambda,x,t)\mbox{ for $\lambda\geq 0$},\notag\\
\tilde {\mathcal U}(\lambda,x,t)&:= {\mathcal U}(-\lambda,x,t)\mbox{  for $\lambda<0$}.
 \end{align}
 Using \eqref{Katoim}, and the fact that $\cH^{1/2}$ commutes with the semigroup $\mathcal{S}$ we deduce that
 \begin{align}\label{space2reg}
\Big(\int_{I}\|\tilde {\mathcal U}(\lambda)\|_{{\E}(\mathbb R^n\times\mathbb R)}^2\, |\lambda|^{1-2s}\d\lambda\Big)^{\frac 12}&\sim \Big(\int_{I}
\bigl (\|\tilde {\mathcal U}(\lambda)\|_{{\L}^2(\mathbb R^n\times\mathbb R)}^2+\|\cH^{1/2}\tilde {\mathcal U}(\lambda)\|_{{\L}^2(\mathbb R^n\times\mathbb R)}^2\bigr )\, |\lambda|^{1-2s}\d\lambda\Big)^{\frac 12}\notag\\
&\leq c (||u||_2+||\cH^{1/2}u||_2)\Big(\int_{I}|\lambda|^{1-2s}\d\lambda\Big)^{\frac 12},
\end{align}
whenever {$u\in \dom(\mathcal{H}^{\bar s})$.} Furthermore, using \eqref{53+--} we have
\begin{align*}
\Big(\int_{I}\left(\|\partial_\lambda \tilde {\mathcal U}\|_{{\L}^2(\mathbb R^n\times\mathbb R)}^2\right)|\lambda|^{1-2s}\, \d\lambda\Big)^{\frac 12}&\leq \Big(\int_{I}\left(\||\lambda|^{1-2s}\partial_\lambda \tilde {\mathcal U}\|_{{\L}^2(\mathbb R^n\times\mathbb R)}^2\right)|\lambda|^{2s-1}\, \d\lambda\Big)^{\frac 12}\notag\\
&\leq c (||u||_2+||\cH^{s+\varepsilon}u||_2)\Big(\int_{I}\max\{1,|\lambda|^{4\varepsilon}\}|\lambda|^{2s-1}\d\lambda\Big)^{\frac 12},
\end{align*}
whenever {$u\in \dom(\mathcal{H}^{\bar s})$. Recall that $s\in (0,1)$, $\varepsilon\in (0,1)$ is small, and  $s+\varepsilon\leq 1$.} In particular,
\begin{align}\label{space2reg+}
\|\tilde {\mathcal U}\|_{W_{1-s}(I,{\E}(\mathbb R^n\times\mathbb R))}\leq c(||u||_2+||\cH^{1/2}u||_2+||\cH^{s+\varepsilon}u||_2)<\infty,
\end{align}
whenever {$u\in \dom(\mathcal{H}^{\bar s})$} and ${I}\subset\mathbb R$, and for a constant depending on $n$, $s$, $\varepsilon$ and $I$. Using \eqref{space2reg+} we see that if $s\in (0,1/2)$, then
$\|\tilde {\mathcal U}\|_{W_{1-s}(I,{\E}(\mathbb R^n\times\mathbb R))}$ is finite if $u\in \dom(\cH^{1/2})$, and if $s\in [1/2,1)$ the same conclusion holds if {$u\in \dom(\cH^{s+\varepsilon})$.} I.e., with $\bar{s}=\max\{1/2,s+\varepsilon\}$, then $\|\tilde {\mathcal U}\|_{W_{1-s}(I,{\E}(\mathbb R^n\times\mathbb R))}$ is finite if $u\in \dom(\cH^{\bar{s}})$.  In particular, using Theorem \ref{thm1} we can conclude that if $u\in \dom(\cH^{\bar{s}})$, then  $\tilde {\mathcal U}$ is a reinforced weak solution to the PDE{
 \begin{align}\label{globaldp+}
\partial_\lambda(|\lambda|^{1-2s}\partial_\lambda  \tilde {\mathcal U}) &=|\lambda|^{1-2s}(\partial_t  \tilde {\mathcal U}  -\div_{x} ( A(x,t)\nabla_{x} \tilde {\mathcal U})),
\end{align}}
in $(\mathbb R\setminus \{0\})\times\mathbb R^n\times \mathbb R$, in the sense that for all finite intervals $\overline{I}\subset\mathbb R\setminus\{0\}$, $\tilde {\mathcal U}$ satisfies \eqref{space2reg} and \eqref{space2reg+},  and
\begin{align}\label{globaldp+aa-}
\int_{\mathbb R}\iint_{\mathbb R^n\times\mathbb R} \bigl (A(x,t) \gradx \tilde {\mathcal U} \cdot \cl{\gradx \Phi} +\HT \dhalf \tilde {\mathcal U} \cdot \cl{\dhalf \Phi}+\partial_\lambda  \tilde {\mathcal U}\cl{\partial_\lambda \Phi}\bigr ) \, |\lambda|^{1-2s}\d x\d t\d\lambda=0,
\end{align}
for all   $\Phi\in C_0^\infty(I\times\mathbb R^n\times\mathbb R)$. Furthermore,
\begin{align}\label{globaldpa}
\mp\lim_{\lambda\to 0^\pm}|\lambda|^{1-2s}\partial_\lambda \tilde {\mathcal U}(\lambda,x,t) &=c_s{\cH}^su(x,t),\notag\\
\lim_{\lambda\to 0^\pm} \tilde {\mathcal U}(\lambda,x,t)&=u(x,t),
\end{align}
where the limits are taken in $H$.

Building on Theorem \ref{thm1},  and the above, we prove the following theorem.
\begin{thm}\label{thm2+} Given $s\in (0,1)$ and  $\varepsilon\in (0,1)$ small, $s+\varepsilon\leq 1$, let $\bar{s}=\max\{1/2,s+\varepsilon\}$. Let $\Omega  \subset\mathbb R^n$ be a domain and let $I,J\subset\mathbb R$ be intervals. Assume that $I=(-R,R)$ for some $R\in (0,\infty)$. Let $u\in \dom({\cH}^{\bar s})$ be a solution to ${\cH}^su(x,t)=0$ in $\Omega  \times J$. Let  $\tilde {\mathcal U}(\lambda,x,t)$ be defined as in \eqref{globaldpakja}.
 Then $\tilde {\mathcal U}$ is a reinforced weak solution to the problem
 {
\begin{align}\label{globaldp+fa}
\partial_\lambda(|\lambda|^{1-2s}\partial_\lambda  \tilde {\mathcal U}) &=|\lambda|^{1-2s}(\partial_t  \tilde {\mathcal U}  -\div_{x} ( A(x,t)\nabla_{x} \tilde {\mathcal U}))\mbox{ in $I\times \Omega  \times J$,}\notag \\
\tilde {\mathcal U}(0,x,t)&={\mathcal U}(0,x,t)=u(x,t)\mbox{ on $\{0\}\times \Omega  \times J$}.
\end{align}}
in the sense that $\tilde {\mathcal U}$ satisfies \eqref{space2reg}, $\tilde {\mathcal U}\in W_{1-s}(I,{\E}(\mathbb R^n\times\mathbb R))$ and
\begin{align}\label{globaldp+aa}
\int_{\mathbb R}\iint_{\mathbb R^n\times \mathbb R} \bigl (A(x,t) \gradx \tilde {\mathcal U} \cdot \cl{\gradx \Phi} +\HT \dhalf \tilde {\mathcal U} \cdot \cl{\dhalf \Phi}+\partial_\lambda  \tilde {\mathcal U}\cl{\partial_\lambda  \Phi}\bigr ) \, |\lambda|^{1-2s}\d x\d t\d\lambda=0,
\end{align}
 for all $\Phi\in C_0^\infty(I\times \Omega\times J)$, and
\begin{align}
\lim_{\lambda\to 0}\tilde {\mathcal U}(\lambda,x,t)&=\lim_{\lambda\to 0}{\mathcal U}(\lambda,x,t)=u(x,t)\mbox{ on $\{0\}\times \Omega  \times J$},
\end{align}
in the sense of limits in ${\L}^2(\Omega  \times J,\d x\d t)$.
\end{thm}

\begin{proof} By the above we know that $\tilde {\mathcal U}$ is a reinforced weak solution in $(\mathbb R\setminus \{0\})\times\mathbb R^n\times \mathbb R$ and that \eqref{globaldpa} holds
for a.e. $(x,t)\in \mathbb R^n\times \mathbb R$ as the limits are taken in $H$. Furthermore, using \eqref{space2reg+} we have $\tilde {\mathcal U}\in W_{1-s}(I,{\E}(\mathbb R^n\times\mathbb R))$, and as $u\in \dom({\cH}^{\bar s})$ is a solution to ${\cH}^su(x,t)=0$ in $\Omega  \times J$ it follows, from the equality on the first line in \eqref{globaldpa}, that
{
\begin{align}\label{globaldpagga+}
\mp\lim_{\lambda\to 0^\pm}\langle |\lambda|^{1-2s}\partial_\lambda \tilde {\mathcal U}(\lambda),v\rangle_H =0,
\end{align}}
for all functions $v\in {\L}^2(\Omega\times J)$. This is essentially the only new information compared to the discussion before the statement of the theorem, and this is the information that we have to exploit. In particular, to prove the theorem it suffices to prove that
\begin{align}\label{globaldp+aagg}
\biggl (\int_{\epsilon}^\infty +\int_{-\infty}^{-\epsilon}\biggr )\iint_{\mathbb R^n\times\mathbb R} \bigl (A(x,t) \gradx \tilde {\mathcal U} \cdot \cl{\gradx \Phi} +\HT \dhalf  \tilde {\mathcal U} \cdot \cl{\dhalf \Phi}+\partial_\lambda  \tilde {\mathcal U}\cl{\partial_\lambda  \Phi}\bigr ) \, |\lambda|^{1-2s}\d x\d t\d\lambda
\end{align}
tends to 0 as $\epsilon\to 0^+$,  whenever $\Phi\in C_0^\infty(I\times \Omega\times J)$. To proceed we first note that if $\lambda\in \mathbb R\setminus\{0\}$, then
{
\begin{align}\label{ident}
\langle(|\lambda|^{1-2s}\tilde{\mathcal{U}}')'(\lambda),\Phi(\lambda)\rangle_H =|\lambda|^{1-2s}\langle \cH \tilde{\mathcal{U}}(\lambda),\Phi(\lambda)\rangle_{V',V},
\end{align}}
and hence we  deduce that
\begin{align*}
& \biggl (\int_{\epsilon}^\infty +\int_{-\infty}^{-\epsilon}\biggr )\iint_{\mathbb R^n\times\mathbb R} \bigl (A(x,t) \gradx \tilde {\mathcal U} \cdot \cl{\gradx \Phi} +\HT \dhalf  \tilde {\mathcal U} \cdot \cl{\dhalf \Phi}+\partial_\lambda  \tilde {\mathcal U}\cl{\partial_\lambda  \Phi}\bigr ) \, |\lambda|^{1-2s}\d x\d t\d\lambda\notag\\
&=-\iint_{\mathbb R^n\times\mathbb R} \partial_\lambda \tilde{\mathcal{U}}(\epsilon)\cl{\Phi(\epsilon)}\,  |\epsilon|^{1-2s}\d x\d t+
\iint_{\mathbb R^n\times\mathbb R} \partial_\lambda \tilde{\mathcal{U}}(-\epsilon)\cl{\Phi(-\epsilon)}\,  |\epsilon|^{1-2s}\d x\d t\notag\\
&=-\iint_{\mathbb R^n\times\mathbb R} \partial_\lambda {\mathcal{U}}(\epsilon)\cl{\tilde\Phi(\epsilon)}\,  |\epsilon|^{1-2s}\d x\d t,
\end{align*}
where $\tilde \Phi(\lambda,x,t)=\Phi(\lambda,x,t)+\Phi(-\lambda,x,t)$. We write
\begin{align}
\iint_{\mathbb R^n\times\mathbb R} \partial_\lambda {\mathcal{U}}(\epsilon)\cl{\tilde\Phi(\epsilon)}\,  |\epsilon|^{1-2s}\d x\d t&=
\iint_{\mathbb R^n\times\mathbb R} \partial_\lambda {\mathcal{U}}(\epsilon)\cl{(\tilde\Phi(\epsilon)-\tilde\Phi(0))}\,  |\epsilon|^{1-2s}\d x\d t\notag\\
&+\iint_{\mathbb R^n\times\mathbb R} \partial_\lambda {\mathcal{U}}(\epsilon)\cl{\tilde\Phi(0)}\,  |\epsilon|^{1-2s}\d x\d t\notag\\
&=:I_1(\epsilon)+I_2(\epsilon).
\end{align}
Using conclusion \eqref{53+--} of  Theorem \ref{thm1} we have
{
\begin{align}\label{53+--a}
\sup_{\lambda\in (0,1)}|||\lambda|^{1-2s}\partial_\lambda \mathcal{U}||_2<\infty.
\end{align}}
Using \eqref{53+--a} and \eqref{globaldpagga+}, and dominated convergence, we see that
\begin{align*}
|I_1(\epsilon)|\to 0\mbox{ as $\epsilon\to 0$},
\end{align*}
and
\begin{align}\label{ident+}
\limsup_{\epsilon\to 0^+} I_2(\epsilon)=0,
\end{align}
as ${\cH}^su(x,t)=0$ a.e. on $\Omega  \times J$. Hence
\begin{align}\label{globaldp+aaggha}
\limsup_{\epsilon\to 0^+}\iint_{\mathbb R^n\times\mathbb R} \partial_\lambda {\mathcal{U}}(\epsilon)\cl{\tilde\Phi(\epsilon)}\,  |\epsilon|^{1-2s}\d x\d t=0,
\end{align}
and we can conclude that the expression in \eqref{globaldp+aagg} tends to 0 as $\epsilon\to 0^+$. In particular,
\begin{align}\label{globaldp+aaggeu}
\int_{\mathbb R}\iint_{\mathbb R^n\times\mathbb R} \bigl (A(x,t) \gradx \tilde {\mathcal U} \cdot \cl{\gradx \Phi} +\HT \dhalf  \tilde {\mathcal U} \cdot \cl{\dhalf \Phi}+\partial_\lambda  \tilde {\mathcal U}\cl{\partial_\lambda  \Phi}\bigr ) \, |\lambda|^{1-2s}\d x\d t\d\lambda=0,
\end{align}
whenever $\Phi\in C_0^\infty(I\times \Omega\times J)$. This completes the proof of the theorem.
\end{proof}

Let  $u\in \dom({\cH}^{\bar s})$ be a solution to
the non-local Dirichlet problem in \eqref{DP}. Let $B$ and $w$ be as in \eqref{Augmatrix}. Given $u$, Theorem \ref{thm2+} implies that there exists a $\tilde {\mathcal U}=\tilde {\mathcal U}(X,t):=\tilde {\mathcal U}(\lambda,x,t)$,  such that $\tilde {\mathcal U}$ satisfies \eqref{space2reg}, $\tilde {\mathcal U}\in W_{1-s}(I,{\E}(\mathbb R^n\times\mathbb R))$, $I=(-R,R)$, $R>0$, and such that
\begin{align}\label{globaldp+fall}
\tilde {\mathcal U}(0,x,t)=u(x,t)\mbox{ on $\{0\}\times \Omega  \times J$},
\end{align}
and such that
\begin{align}
{\iiint_{\mathbb{R}^n\times\mathbb R} \bigl (B(X,t) \nabla_{X} \tilde {\mathcal U} \cdot \cl{\nabla_{X} {\Phi}} +\HT \dhalf \tilde {\mathcal U} \cdot \cl{\dhalf {\Phi}}\bigr ) \, w(X)\d X\d t=0,}
\end{align}
for all $ {\Phi}\in C_0^\infty(\tilde\Omega\times J)$. Here $X:=(\lambda,x)=(x_0,x)\in\mathbb R^{n+1}$ {and} $\tilde\Omega:=I\times \Omega \subset\mathbb R^{n+1}$. As above, and following \cite{AEN}, we refer to $ \tilde {\mathcal U}$ as a reinforced weak solution to
\begin{eqnarray}\label{eq1deg+}
\mathcal{L}\, \tilde {\mathcal U} := \div_{X} ( w(X)B(X,t)\nabla_{X} \tilde {\mathcal U} )-w(X)\partial_t \tilde {\mathcal U}  = 0,\ (X,t)\in \tilde\Omega\times J.
 \end{eqnarray}
  The solution is referred to as reinforced because \eqref{space2reg} and $\tilde {\mathcal U}\in W_{1-s}(I,{\E}(\mathbb R^n\times\mathbb R))$ encode more regularity in the $t$-variable than what is usually demanded in the weak formulation
 of second order parabolic equations and systems.  Let $\Hdot^{1/2}(\R)$ be the homogeneous Sobolev space of order 1/2. In Section 3 in \cite{AEN} the properties of this space is reviewed and if  $u\in \Hdot^{1/2}(\R)$ and $\phi\in \C_0^\infty(\R)$, then the formula
\begin{align*}
 \int_{\R} \HT\dhalf u\cdot \overline{\dhalf\phi}\d t = - \int_{\R} u \cdot \cl{\partial_{t}\phi} \d t
\end{align*}
holds, where on the right-hand side we use the duality form between $\Hdot^{1/2}(\R)$ and its dual $\Hdot^{-1/2}(\R)$ extending the complex inner product of $\L^2(\R)$. In particular,
\begin{align}\label{globaldp+aaha}
{\iiint_{\mathbb{R}^n\times\mathbb R} \HT\dhalf \tilde {\mathcal U} \cdot \cl{ \dhalf \Phi} \, w(X)\d X\d t} =-\iiint_{\tilde\Omega\times J} \tilde {\mathcal U} \cdot \cl{\partial_t \Phi} \, w(X)\d X\d t,
\end{align}
for all $\Phi\in C_0^\infty(\tilde\Omega\times J)$. Hence,  we can conclude that the reinforced weak solution $\tilde {\mathcal U}$ is a weak solution in the usual sense on $\tilde\Omega\times J$, i.e. $\tilde {\mathcal U}\in \L^2(J, \W^{1,2}(\tilde \Omega))$ and
\begin{align}\label{globaldp+aabb}
\iiint_{\tilde\Omega\times\mathbb R} \bigl (B(X,t) \nabla_{X} \tilde {\mathcal U} \cdot \cl{\nabla_{X} \Phi} -\tilde {\mathcal U} \cdot \cl{\partial_t \Phi}\bigr ) \, w(X)\d X\d t=0,
\end{align}
for all $\Phi\in C_0^\infty(\tilde\Omega\times J)$. This implies that $\pd_{t}\tilde {\mathcal U} \in \L^2(J, \W^{-1,2}(\tilde \Omega))$ and
 \begin{eqnarray}\label{eq1deg+a}
\mathcal{L}\, \tilde {\mathcal U} = \div_{X} ( w(X)B(X,t)\nabla_{X} \tilde {\mathcal U} )-w(X)\partial_t \tilde {\mathcal U}  = 0
 \end{eqnarray}
 on $\tilde\Omega\times J$ in the (traditional) weak sense.

\section{Parabolic equations with real coefficients: local regularity}\label{sec4}

Using Theorem \ref{thm2+} and the subsequent discussion as a starting point, and making  the additional assumption that $A=A(x,t)=\{A_{i,j}(x,t)\}_{i,j=1}^{n}$ is real and measurable, we can derive local properties of solutions to the non-local Dirichlet problem introduced in \eqref{DP}. We continue to denote points in Euclidean $ (n+1) $-space
$ \mathbb R^{n+1} $ by $ (x,t) = ( x_1,
 \dots,  x_n,t)$,  where $ x = ( x_1, \dots,
x_{n} ) \in \mathbb R^{n } $ and $t$ represents the time-coordinate.   Given $(x,t), (y,s)\in\mathbb R^{n+1}$, $r>0$, we let $d (x,t,y,s)$ and $Q_r ( x, t )$ be as introduced before the statement of Theorem \ref{Holder}. As stated, we use lowercase letters $x,y$ to denote points in $\mathbb R^n$. Similarly, we use capital letters $X,Y$ to denote points in $\mathbb R^{n+1}$. In particular, points in Euclidean $ (n+2) $-space
$ \mathbb R^{n+2} $ are denoted by $ (X,t) = (x_0, x_1,
 \dots,  x_n,t)=(\lambda, x_1,
 \dots,  x_n,t)$. The notation $d$ and $Q_r$ generalize immediately from $ \mathbb R^{n+1} $ to $ \mathbb R^{n+2} $,  and we only use capital letters $X,Y$ to emphasize that we
 work in $ \mathbb R^{n+1} $ space wise.

  Assume that $A=A(x,t)=\{A_{i,j}(x,t)\}_{i,j=1}^{n}$ is real, measurable, and satisfies \eqref{ellip}.  Assume that {$u\in \dom({\cH}^{\bar s})$} is a solution to $\cH^s u=0$ in
${Q}_{4r}(z_0,\tau_0)$ in the sense of Definition \ref{solu}. Using Theorem \ref{thm2+} and the subsequent discussion we can conclude that there exists a (traditional) weak solution
 $\tilde {\mathcal U}$ to the equation stated in \eqref{eq1deg+aka} in ${Q}_{4r}(Z_0,\tau_0)$, $Z_0=(0,z_0)$, such that $\tilde {\mathcal U}(0,x,t)=u(x,t)$ for  a.e. $(x,t)\in {Q}_{4r}(z_0,\tau_0)$. $\tilde {\mathcal U}$ is
 defined in \eqref{globaldpakja} based on $\mathcal U$ as in \eqref{51ll}.

 Recall
the construction of $\mathcal{S}$ in Section \ref{sec2-}. By \eqref{hille} we have
\begin{align}\label{Yo5}
\lim_{m\to\infty}\|(\mathcal{S}(r)-R_m(r))u\|_{2}\to 0 \qquad (u \in H),
\end{align}
for every $r\geq 0$ fixed and {where $R_{m}(\lambda)$ was introduced} in \eqref{hille-}. We introduce, for $\m\in \mathbb Z_+$,
\begin{align}\label{51llrep}
\mathcal U_m(X,t):=\mathcal U_m(\lambda,x,t):=\frac{1}{\Gamma(s)}\biggl (\frac\lambda 2\biggr )^{2s}\int_0^\infty r^{-s} e^{-\frac{\lambda^2}{4r}}R_m(r)u(x,t) \, \frac{\d r}{r}.
\end{align}
We first prove the following approximation lemma.

\begin{lem}\label{lemmaEst}
Consider  ${Q}_{4r}(Z_0,\tau_0)\subset\mathbb R^{n+2}$, $Z_0=(0,z_0)$. Then
\begin{align*}
\lim_{m\to \infty}\iiint_{{Q}_{4r}(Z_0,\tau_0)} |\mathcal U(X,t)-\mathcal U_{m}(X,t)|^2\, \d X\d t=0.
\end{align*}
\end{lem}
\begin{proof}  We have
$$\mathcal U(\lambda,x,t)-\mathcal U_m(\lambda,x,t)=\frac{1}{\Gamma(s)}\biggl (\frac\lambda 2\biggr )^{2s}\int_0^\infty r^{-s} e^{-\frac{\lambda^2}{4r}}(\mathcal{S}(r)-R_m(r))u(x,t) \, \frac{\d r}{r}.$$
Using that $\cH$ is maximal accretive we have that
$$||R_m(r)||_{H\to H}\leq 1\mbox{ for all $r>0$, $m\in\mathbb Z_+$}.$$
Hence, letting
$$F_m(\lambda):=\iint_{\mathbb R^n\times \mathbb R} |\mathcal U(\lambda,x,t)-\mathcal U_m(\lambda,x,t)|^2\, \d x\d t,$$
we see that
\begin{align}\label{51ll+}
 F_m(\lambda)\leq \biggl (\frac{1}{\Gamma(s)}\biggl (\frac\lambda 2\biggr )^{2s}\biggr )^2\biggl (\int_0^\infty r^{-s} e^{-\frac{\lambda^2}{4r}}
 \|(\mathcal{S}(r)-R_m(r))u\|_2\, \frac{\d r}{r}\biggr )^2\leq 2\|u\|_2^2,
\end{align}
for all $\lambda>0$ and $m\in\mathbb Z_+$ by the same argument as in \eqref{51apa}. Let $J\subset\mathbb R$ {be} a finite interval and consider $\lambda\in J$. Using \eqref{51ll+}, \eqref{Yo5} and dominated convergence,
\begin{align}\label{51ll+fi}
 \lim_{m\to\infty} F_m(\lambda)&\leq \lim_{m\to\infty}\biggl (\frac{1}{\Gamma(s)}\biggl (\frac\lambda 2\biggr )^{2s}\biggr )^2\biggl (\int_0^\infty r^{-s} e^{-\frac{\lambda^2}{4r}}
 \|(\mathcal{S}(r)-R_m(r))u\|_2\, \frac{\d r}{r}\biggr )^2\notag\\
 &=\biggl (\frac{1}{\Gamma(s)}\biggl (\frac\lambda 2\biggr )^{2s}\biggr )^2\biggl (\int_0^\infty r^{-s} e^{-\frac{\lambda^2}{4r}}\biggl [\lim_{m\to\infty}\
 \|(\mathcal{S}(r)-R_m(r))u\|_2\biggr ]\, \frac{\d r}{r}\biggr )^2=0.
\end{align}
Using this we can conclude that
\begin{align}\label{51ll++}
\lim_{m\to \infty}\int_J\iint_{\mathbb R^n\times \mathbb R} |\mathcal U(\lambda,x,t)-\mathcal U_m(\lambda,x,t)|^2\, \d x\d t\d\lambda=\lim_{m\to \infty}\int_J F_m(\lambda)\, \d\lambda=0.
\end{align}
Hence, if we consider ${Q}_{4r}(Z_0,\tau_0)\subset\mathbb R^{n+2}$, $Z_0=(0,z_0)$, then
\begin{align}\label{51ll++l+}
\lim_{m\to \infty}\iiint_{{Q}_{4r}(Z_0,\tau_0)} |\mathcal U(X,t)-\mathcal U_{m}(X,t)|^2\, \d X\d t=0.
\end{align}
The proof of the lemma is complete.
\end{proof}

To estimate $\mathcal U_{m}$ we will use the following lemma.

\begin{lem} \label{resolvent kernel}
For $\sigma> 0$ and $m \ge 1$, the resolvent $(1+\sigma^{-1} \cH)^{-m}$, defined as a bounded operator on $\L^2(\ree)$, is represented by an non-negative integral kernel $K_{\sigma,m}$ with pointwise bounds
\begin{equation}
\label{eq:gaussian}
  |K_{\sigma,m} (x,t,y,s)| \leq {C{\sigma^{m}} { \chi_{(0,\infty)}}(t-s)}(t-s)^{-n/2+m-1}  {e}^{-{\sigma(t - s)}} {e}^{-c \frac{|x-y|^2}{t-s}},
\end{equation}
where $C,c>0$ depend only on $n$, the ellipticity constants and $m$. Furthermore,
\begin{equation}\label{bla}
\lim_{R\to \infty}\iint_\ree K_{\sigma,m}(x,t,y,s) \chi_{Q_R(0,0)}(y,s) \d y \d s=1,
\end{equation}
where $\chi_{Q_R(0,0)}$ is the indicator function for the parabolic cube $Q_R(0,0)$.
\end{lem}

\begin{proof} This is Lemma 4.3 in \cite{AEN1}.
{We include the proof for completeness.} It suffices to prove the lemma when $m=1$ as iterated convolution in $(x,t)$ of the estimate on the right hand side of \eqref{eq:gaussian} with $m=1$ yields the result. Let $f \in \C^\infty_{0}(\ree)$. Let $u=(1+\sigma^{-1}\cH)^{-1}f$. Then $u\in \L^2(\ree)$ and, in particular, $u$ is a weak solution to $\sigma^{-1}\partial_{t}u-\sigma^{-1} \divx A\gradx u + u= f$. On the other hand, by ~\cite{A} the operator $\cH$ has a fundamental solution, denoted by $K(x,t,y,s)$, having bounds
\begin{eqnarray*}
|K(x,t,y,s)| \leq {C} {\chi_{(0,\infty)}}(t-s) \cdot (t-s)^{-n/2}  {e}^{-c \frac{|x-y|^2}{t-s}}  \qquad \mbox{for $x,y \in \R^n$, $t,s \in \R$},
\end{eqnarray*}
with constants $C,c$ depending only on dimension and the ellipticity constants, and satisfying
\begin{eqnarray}
\label{conservation}
\int_{\R^n} K(x,t,y,s) \, \d y = 1 \qquad \mbox{for $x \in \R^n$, $t,s \in \R$, $t>s.$}
\end{eqnarray}
Furthermore, $K(x,t,y,s)$ is non-negative. Set $K_{\sigma,1}(x,t,y,s)=  \sigma K(x,t,y,s) {e}^{-{\sigma(t - s)}}$ and $$v(x,t)= \iint_\ree K_{\sigma,1}(x,t,y,s) f(y,s) \d y \d s.$$
Obviously also $K_{\sigma,1}(x,t,y,s)$ is non-negative. The estimate on $K$ implies $v\in \L^2(\ree)$ and a calculation shows that
$v$ is a weak solution to the same equation as $u$. Thus, $w:=u-v$ is a weak solution of
$\partial_{t}w-\divx A\gradx w + \sigma w= 0$ and we may use the Caccioppoli estimate  in Lemma 2.1 in \cite{AEN1} in $\ree$. Choosing test functions $\psi$ that converge to $1$ reveals $\gradx w=0$ as $w\in \L^2(\ree)$. Hence $w$ depends only on $t$. Again, as $w\in \L^2(\ree)$, $w$ must be $0$. This shows that $(1+\sigma^{-1}\cH)^{-1}f$ has the desired representation for all $f \in \C^\infty_{0}(\ree)$ and we conclude by density. \eqref{bla} follows readily from the construction of $K_{\sigma,1}$.\end{proof}

Using the previous two lemmas we next prove two lemmas containing estimates for $\mathcal U$.
\begin{lem}\label{crucial} Consider ${Q}_{4r}(Z_0,\tau_0)\subset\mathbb R^{n+2}$, $Z_0=(0,z_0)$, and assume that $$\|u\|^2_{{\L}^\infty(\mathbb R^n\times (-\infty,\tau_0])}<\infty.$$ Then
\begin{align}\label{bound+}
\bariiint_{{Q}_{4r}(Z_0,\tau_0)} |\mathcal U(X,t)|^2\, \d X\d t\leq \|u\|^2_{{\L}^\infty(\mathbb R^n\times (-\infty,\tau_0])}.
\end{align}
\end{lem}
\begin{proof} Using Lemma \ref{resolvent kernel} we see that, if $u\in H$, then
\begin{align}\label{51ll++ahaaa}
R_m(r)u(x,t)&= (I+(m/r)^{-1}\cH)^{-m}u(x,t)= \iint_\ree K_{m/r,m}(x,t,y,s) u(y,s) \d y \d s.
\end{align}
Furthermore, as
\begin{align}\label{51ll++ahaaa2}
 &\bigl |\iint_\ree K_{m/r,m}(x,t,y,s) u(y,s) \d y \d s\bigr |\notag\\
 &=\bigl |\lim_{R\to \infty}\iint_\ree K_{m/r,m}(x,t,y,s) u(y,s)\chi_{Q_R(0,0)}(y,s) \d y \d s\bigr |\notag\\
 &\leq \|u\|_{{\L}^\infty(\mathbb R^n\times (-\infty,t])}\lim_{R\to \infty}\iint_\ree K_{m/r,m}(x,t,y,s)\chi_{Q_R(0,0)}(y,s)\d y \d s\notag\\
 &=\|u\|_{{\L}^\infty(\mathbb R^n\times (-\infty,t])},
\end{align}
we have, for  $r>0$ and $m\in\mathbb Z_+$,  that
\begin{align}\label{51ll++ahaaa3}
|R_m(r)u(x,t)|&\leq \|u\|_{{\L}^\infty(\mathbb R^n\times (-\infty,t])}.
\end{align}
Using this, and the definition of $\mathcal U_{m}$, we deduce that
\begin{align}\label{bound}
|\mathcal U_{m}(X,t)|\leq \|u\|_{{\L}^\infty(\mathbb R^n\times (-\infty,t])}\mbox{ for $(X,t)\in \mathbb R^{n+2}$}.
\end{align}
By Lemma \ref{lemmaEst} we have
\begin{align}
\lim_{m\to \infty}\iiint_{{Q}_{4r}(Z_0,\tau_0)} |\mathcal U(X,t)-\mathcal U_{m}(X,t)|^2\, \d X\d t=0.
\end{align}
Hence,
\begin{align}\label{51ll++l+-}
\iiint_{{Q}_{4r}(Z_0,\tau_0)} |\mathcal U(X,t)|^2\, \d X\d t\leq \limsup_{m\to \infty}\iiint_{{Q}_{4r}(Z_0,\tau_0)} |\mathcal U_{m}(X,t)|^2\, \d X\d t.
\end{align}
Using \eqref{51ll++l+-} and \eqref{bound},
\begin{align}
\bariiint_{{Q}_{4r}(Z_0,\tau_0)} |\mathcal U(X,t)|^2\, \d X\d t\leq \|u\|^2_{{\L}^\infty(\mathbb R^n\times (-\infty,\tau_0])}.
\end{align}
The proof is complete.
\end{proof}

Based on $ \mathcal U_{m}(\lambda,x,t)$ as above, define $\tilde {\mathcal U}_{m}$ on $\mathbb R\times\mathbb R^n\times \mathbb R$ through
\begin{align}\label{globaldpakjap}
\tilde {\mathcal U}_{m}(\lambda,x,t)&:= \mathcal U_{m}(\lambda,x,t)\mbox{ for $\lambda\geq 0$},\notag\\
\tilde {\mathcal U}_{m}(
\lambda,x,t)&:= \mathcal U_{m}(-\lambda,x,t)\mbox{  for $\lambda<0$}.
 \end{align}

\begin{lem}\label{crucial+} Consider ${Q}_{4r}(Z_0,\tau_0)\subset\mathbb R^{n+2}$, $Z_0=(0,z_0)$.   Assume that $u\geq 0$ on $\mathbb R^n\times (-\infty,\tau_0]$. Then
\begin{align}\label{bound+max}
\tilde {\mathcal U}_{m}(X,t)\geq 0\mbox{ on $\mathbb R\times \mathbb R^n\times\mathbb R$,}
\end{align}
and
\begin{align}\label{51ll++l+ha}
\lim_{m\to \infty}|\{(X,t)\in {Q}_{4r}(Z_0,\tau_0):\ |\tilde{\mathcal U}(X,t)-\tilde {\mathcal U}_{m}(X,t)|>\epsilon\}|=0,
\end{align}
{for all $\epsilon>0$.}
\end{lem}
\begin{proof} \eqref{bound+max} is a consequence of \eqref{51ll++ahaaa} and non-negativity of the kernels $\{K_{\sigma,m}(x,t,y,s)\}$ of Lemma \ref{resolvent kernel}. \eqref{51ll++l+ha} follows immediately from Lemma \ref{lemmaEst}.
\end{proof}

\subsection{Proof of Theorem \ref{Holder}} Using Theorem \ref{thm2+} and the subsequent discussion we can conclude that there exists a (traditional) weak solution
 $\tilde {\mathcal U}$ to the problem in \eqref{eq1deg+aka}
 in ${Q}_{4r}(Z_0,\tau_0)$, $Z_0=(0,z_0)$, such that $\tilde {\mathcal U}(0,x,t)=u(x,t)$ for  a.e. $(x,t)\in {Q}_{4r}(z_0,\tau_0)$. Using \cite{CF,CS} we see that
 there exist constants $c$, $1\leq c<\infty$, and $\alpha\in (0,1)$, both depending only on the structural constants $n$, $c_1$, $c_2$, and $s$ and $\bar{s}$, such that
\[
|{\tilde {\mathcal U}(X,t)-\tilde {\mathcal U}(Y,s)}|\le c\left(\frac{d (X,t,Y,s)}{r}\right)^\alpha\sup_{{Q}_{2r}(Z_0,\tau_0)}|\tilde {\mathcal U}|,
\]
whenever $(X,t)$, $(Y,s)\in {Q}_{r}(Z_0,\tau_0)$. Hence
\[
|{u(x,t)-u(y,s)}|\le c \left(\frac{d (x,t,y,s)}{r}\right)^\alpha\sup_{{Q}_{2r}(Z_0,\tau_0)}|\tilde {\mathcal U}|,
\]
whenever $(x,t)$, $(y,s)\in {Q}_{r}(z_0,\tau_0)$. Furthermore, using \cite{CF,CS} we also have
\begin{align}\label{eq1deg+aka+}\sup_{{Q}_{2r}(Z_0,\tau_0)}|\tilde {\mathcal U}|&\leq c\biggl (\bariiint_{{Q}_{4r}(Z_0,\tau_0)}|\tilde {\mathcal U}|^2\, \d X\d t\biggr )^{1/2}.
 \end{align}
 Using Lemma \ref{crucial}
\begin{align}\label{eq1deg+aka+bb}\biggl (\bariiint_{{Q}_{4r}(Z_0,\tau_0)}|\tilde {\mathcal U}|^2\, \d X\d t\biggr )^{1/2}\leq \|u\|_{{\L}^\infty(\mathbb R^n\times (-\infty,\tau_0])}.
 \end{align}
 Using \eqref{eq1deg+aka+bb}
 \[
|{u(x,t)-u(y,s)}|\le c\left(\frac{d (x,t,y,s)}{r}\right)^\alpha \|u\|_{{\L}^\infty(\mathbb R^n\times (-\infty,\tau_0])},
\]
whenever $(x,t)$, $(y,s)\in {Q}_{r}(z_0,\tau_0)$, and this completes the proof.

\subsection{Proof of Theorem \ref{Harnack}} Let $\tilde {\mathcal U}$ and $\tilde {\mathcal U}_{m}$ be as above.  As seen in the proof of Theorem \ref{Holder}, $\tilde {\mathcal U}$ is continuous in ${Q}_{4r}(Z_0,\tau_0)$. Consider $(X,t)\in {Q}_{3r}(Z_0,\tau_0)$ and assume that $\tilde {\mathcal U}(X,t)<0$. Then by Theorem \ref{Holder} there exists
 $\rho>0$ such that $\tilde {\mathcal U}(Y,s)<\delta<0$ for all $(Y,s)\in {Q}_{\rho}(X,t)$. However, this {contradicts} the conclusions of Lemma \ref{crucial+}.  Hence $\tilde {\mathcal U}(X,t)\geq 0$ when $(X,t)\in {Q}_{3r}(Z_0,\tau_0)$. Applying to
 $\tilde {\mathcal U}$ the Harnack inequality proved in  \cite{CS} we see that there exist a constant $c$, $1\leq c<\infty$, depending only on the structural constants $n$, $c_1$, $c_2$, and $s$ and $\bar{s}$, such that
\begin{align*}
\sup_{Q_{2r}^-(Z_0,\tau_0)} \tilde {\mathcal U}
\leq c \inf_{Q_{2r}^+(Z_0,\tau_0)} \tilde {\mathcal U},
\end{align*}
where
\begin{align*}
Q_{2r}^-(Z_0,\tau_0)&:={Q_{2r}(Z_0,\tau_0)}\cap \{(X,t):\ \tau_0-3r^2/4<t< \tau_0-r^2/2\},\notag\\
Q_{2r}^+(Z_0,\tau_0)&:={Q_{2r}(Z_0,\tau_0)}\cap \{(X,t):\ \tau_0-r^2/4<t< \tau_0\}.
\end{align*}
The stated Harnack inequality for $u$ now follows immediately from this Harnack inequality for $\tilde {\mathcal U}$.

\begin{rem} The results established in \cite{CF,CS} are stated under the assumption of symmetric coefficients. However, analyzing the proofs
of the De Giorgi-Moser-Nash arguments one deduces that symmetry is not an issue and the conclusions hold assuming only that $A=A(x,t)=\{A_{i,j}(x,t)\}_{i,j=1}^{n}$ is real, measurable, and satisfies \eqref{ellip}.
\end{rem}

\section{Concluding remarks}\label{conc}
We believe that our paper represents a step towards a regularity theory for fractional powers of parabolic operators with time-dependent, bounded and measurable coefficients. Indeed, the results presented are not final as we, when considering $\cH^s$, frequently assume stronger regularity compared to $u\in\dom(\cH^s)$. Let us first remark that in the special case $s=1/2$ the results presented can be sharpened. Indeed, $e^{-\lambda\sqrt{\cH}}$ is well-defined as $\cH$ is maximal accretive, and if
$u\in\dom(\cH^{1/2})$ then $\ru(\lambda):=e^{-\lambda\sqrt{\cH}}u$ is a reinforced weak solution to the PDE in \eqref{globaldp+}
in $(\mathbb R\setminus \{0\})\times\mathbb R^n\times \mathbb R$.  In particular, using this extension the conclusions of Theorem \ref{thm2+} hold in the case of $s=1/2$ for all
$u\in\dom(\cH^{1/2})$. Therefore, in this case we can conclude that Theorem \ref{Holder} and Theorem \ref{Harnack} holds for $u\in\dom(\cH^{1/2})$ such that $\cH^{1/2}u=0$, and with constants which only depend on the structural constants $n$, $c_1$, $c_2$. In the case $s\in (0,1/2)$, Theorem \ref{Holder} and Theorem \ref{Harnack} are proven
for $u\in\dom(\cH^{1/2})$ such that $\cH^{s}u=0$. I.e., in this case we assume considerably more a priori regularity on $u$ compared to what is needed for the formulation of
$\cH^{s}u=0$. On the other hand, by the solution of the parabolic Kato problem, $\dom(\cH^{1/2})$ has an explicit description. In the case $s\in (1/2,1)$, Theorem \ref{Holder} and Theorem \ref{Harnack} are proven
for $u\in\dom(\cH^{s+\varepsilon})$ such that $\cH^{s}u=0$ and where $\varepsilon>0$ can be chosen arbitrary small, but fixed. We need this slight extra regularity to conclude
\eqref{53+--}. Considering $\cH^{s}$ we would like to only assume $u\in\dom(\cH^{s})$ to make conclusions,  but currently we do not know how to accomplish this  due to the weak properties of the semigroup we use to go from the fractional powers to the extension problem.  For comparison, in a previous version of this paper we simply proved all our results assuming $u\in\dom$, but as discussed it is an open problem to completely decipher the condition $u\in\dom$, and it is hard to pinpoint in explicit terms what a priori regularity on $u$ is assumed in this case. We believe this to be a key problem for future research. In addition, we believe that it may be interesting to further understand what insights semigroup theory, subordination and Bochner’s functional calculus can bring to the topic.\\

\noindent
{\bf Acknowledgement.} The authors like to thank an anonymous referee for a very careful reading of the paper and for several valuable suggestions. The second author would also like to thank Moritz Egert for some valuable comments.

\def\cprime{$'$} \def\cprime{$'$} \def\cprime{$'$}

\end{document}